\documentclass[final]{siamltex}

\usepackage{amssymb,amsmath,amsfonts,graphicx,subfigure,pstricks,color}

\newtheorem{remark}[theorem]{Remark}
\newtheorem{example}[theorem]{Example}

\newcommand{\R}{\mathbb{R}}

\newcommand{\ol}[1]{\overline{#1}}

\newcommand{\dd}{\text{d}}

\newcommand{\norm}[1]{\Vert #1 \Vert}
\newcommand{\abs}[1]{\vert #1 \vert}

\title{Exact Relaxation for Classes of Minimization Problems with Binary Constraints\thanks{M.H. acknowledges support by the German Research Fund (DFG) through the Research Center MATHEON Project C28 and the SPP 1253 "Optimization with Partial Differential Equations", the Austrian Science Fund (FWF) through the START Project Y 305-N18 "Interfaces and Free Boundaries" and the SFB Project F32 04-N18 "Mathematical Optimization and Its Applications in Biomedical Sciences", and support through a J. Tinsley Oden Fellowship at the Institute for Computational Engineering and Sciences (ICES) at UT Austin, Texas, USA. M.B. acknowledges support by the German Research Fund (DFG) through the projects BU 2327/1 and BU 2327/6 and the SFB 656/B2.}}

\author{Martin Burger \thanks{Institut f\"ur Numerische und Angewandte Mathematik, Westf\"alische Wilhelms-Universit\"at M\"unster, Einsteinstr. 62, 48149 M\"unster, Germany. (martin.burger@wwu.de)}
\and Yiqiu Dong \thanks{Institute of Biomathematics and Biometry, Helmholtz Zentrum M\"unchen, Ingolst\"adter Landstrasse 1, 85764 Neuherberg, Germany. (yiqiu.dong@helmholtz-muenchen.de)}
\and Michael Hinterm\"uller \thanks{Department of Mathematics, Humboldt-University of Berlin, Unter den Linden 6, 10099 Berlin, Germany, and START-Project ``Interfaces and Free Boundaries'' and SFB ``Mathematical Optimization and Applications in Biomedical Science'', Institute of Mathematics and Scientific Computing, University of
Graz, Heinrichstrasse 36, A-8010 Graz, Austria. (hint@math.hu-berlin.de)}}
\begin{document}
\maketitle

\begin{abstract}
Relying on the co-area formula, an exact relaxation framework for minimizing objectives involving the total variation of a binary valued function (of bounded variation) is presented. The underlying problem class covers many important applications ranging from binary image restoration, segmentation, minimal compliance topology optimization to the optimal design of composite membranes and many more. The relaxation approach turns the binary constraint into a box constraint. It is shown that thresholding a solution of the relaxed problem almost surely yields a solution of the original binary-valued problem. Furthermore, stability of solutions under data perturbations is studied, and, for applications such as structure optimization, the inclusion of volume constraints is considered. For the efficient numerical solution of the relaxed problem, a locally superlinearly convergent algorithm is proposed which is based on an exact penalization technique, Fenchel duality, and a semismooth Newton approach. The paper ends by a report on numerical results for several applications in particular in mathematical image processing.
\end{abstract}



\section{Introduction} 
We consider the binary constrained minimization problem
\begin{equation}
\min_{(u,v)\in BV(\Omega;\{0,1\})\times\mathcal{V}} H(u,v):=F(v) + \int_\Omega G(v)~u~\dd x + \beta~ J(u),\\
\label{full_model}
\end{equation}
where ${\cal V}$ is a reflexive real Banach space, $F: {\cal V} \rightarrow \R$ is a (sequentially) weakly lower semicontinuous functional, which is 
bounded from below, $G:{\cal V} \rightarrow L^p(\Omega)$ is a strongly continuous nonlinear operator for 
some $p > 1$, $\Omega\subset\mathbb{R}^d$, $d\in\mathbb{N}$, is a bounded and piecewise smooth open set. By $BV(\Omega;S)$ we denote the space of functions of bounded variation, where $(S;\delta)$ is a locally compact metric space with some set $S$ and a metric $\delta$ on $S$. According to \cite{MR1079985}, a Borel function $u:\Omega\to S$ is an element of $BV(\Omega;S)$ if there exists a finite measure $\mathfrak{m}$ such that $\mathfrak{m}(B)\geq�\int_B|D\varphi(u)|=|D\varphi(u)|(B)$ for all Borel sets $B\subset\Omega$ and for any Lipschitz function $\varphi:S\to\mathbb{R}$ with a Lipschitz constant not larger than $1$, i.e., $\varphi\in\operatorname{Lip}_1(S)$.
Here, $|D\varphi(u)|(\cdot)$ represents the usual total variation measure of a function $\varphi(u)\in BV(\Omega)$; see, e.g., \cite{BVBook2}. In fact, it is well-known that $w\in BV(\Omega)$ iff $w\in L^1(\Omega)$ and the total variation measure ($BV$-seminorm)
\[
\int_\Omega |Dw|=\sup\bigg\{\int_\Omega w\;\mathrm{div} \vec{v} \;
dx:\vec{v}\in(C_0^\infty(\Omega))^d,\|\vec{v}\|_{\infty}\leq1\bigg\}
\]
is finite. In the definition of $\int_\Omega |Dw|$, $(C_0^\infty(\Omega))^d$ is the space of $\mathbb{R}^d$-valued functions with compact support in $\Omega$. The space $BV(\Omega)$ endowed with the norm $\|w\|_{BV(\Omega)}=\|w\|_{L^1(\Omega)}+\int_\Omega |Dw|$ is a Banach space; see, e.g., \cite{BVBook2}.  
In \eqref{full_model}, we have $(S;\delta):=(\{0,1\};|\cdot|)$ and the total variation (TV) of $u\in BV(\Omega;S)$ is given by $$J(u):=\int_\Omega |Du|:=\mathcal{M}-\sup\left\{|D\varphi (u)|:\varphi\in\operatorname{Lip}_1(S)\right\},$$ which is coupled to the remaining terms of the objective by $\beta>0$. Here, $\mathcal{M}-\sup$ defines the supremum over $\sigma$-additive measures over the Borel sets in $\Omega$; see \cite{MR1079985}.


Problems of the type considered in \eqref{full_model} are of fundamental importance in many application areas. Indeed, in imaging science, both the TV-based binary image denoising model \cite{TV1,TVmodel} and the piecewise constant Mumford-Shah segmentation model \cite{Chan_Vese,Mumford-Shah} are respectively equivalent to \eqref{full_model} with $F\equiv0$ and $G(\cdot)=g\in L^p(\Omega)$ for $p>1$, where the binary variable $u$ represents the binary image, i.e. the support set of image features, which is recovered from noisy data contained in $g$. In the optimal design of composite membranes \cite{OD_1,OD_2}, the minimization problem of the first eigenvalue for an underlying elliptic problem fits into the structure considered in \eqref{full_model}, where $v$ is the corresponding eigenfunction and the binary-valued $u$ represents the optimal configuration. Furthermore, \eqref{full_model} also appears in the optimization of the topology of structures, e.g., in the well-known minimal compliance problem; see, e.g., \cite{TOP_Book2,TOP_Book}. In the latter case, $v$ denotes the elastic displacement and $u$ the material density.  

Clearly, the binary constraint on $u$, contained implicitly in the definition of the underlying function space, poses significant analytical as well as numerical challenges. In fact, the binary constraint renders \eqref{full_model} non-convex
and combinatorial in nature. Combined with the non-differentiability of $H$ in the first variable (which is due to the presence of $J$), the latter property challenges the development of efficient numerical solvers. Due to these mathematical difficulties and the wide applicability, recently problems with binary constraints have become the focus of research efforts in different model frameworks; see, e.g., \cite{Filter_TOP,analysisL1TV1,CEM-2006,OD_1}. 

In this paper, starting from the general model \eqref{full_model} we propose a relaxation framework for \eqref{full_model} by replacing $u(x)\in\{0,1\}$ by the box constraint $u(x)\in[0,1]$. Interestingly, it turns out that this relaxation is exact, in the sense that a minimizer of \eqref{full_model} can be obtained almost surely from thresholding a solution to the relaxed problem with $u(x)\in[0,1]$ for almost all $x\in\Omega$. We also provide a stability result with respect to perturbations of the operators $G$, and we incorporate volume constraints, which are important, e.g., in topology optimization. Numerically, the box constraint on $u$ is realized by an exact penalization, and considering the nondifferentiable TV-term in the objective, a Fenchel duality technique is utilized to solve the relaxed problem on $u$. Specifically, the pre-dual to the relaxation is a quadratic optimization problem with box constraints, which is solved by an inexact semismooth Newton method. The resulting algorithm converges locally at a superlinear rate, and the efficiency of the proposed solver is demonstrated by means of several applications of \eqref{full_model}.

The outline of this paper is as follows. In Section 2, we study the exact relaxation of \eqref{full_model} and stability results. Moreover, the situation with an additional volume constraint is considered. Section 3 describes our algorithmic framework. Several applications including numerical results are shown in Section 4, and, finally, conclusions are drawn in Section 5.

\section{Exact relaxation in BV-space}

We start our analysis of \eqref{full_model} by focusing on the binary constrained variable $u$ and by considering an auxiliary problem. After its analysis we return to the full model considered in \eqref{full_model}.

\subsection{An auxiliary problem}
The auxiliary problem, which simplifies \eqref{full_model}, consists of a linear term and the TV regularization only, i.e., we have
\begin{equation}
\min\limits_{u\in BV(\Omega;\{0,1\})} \int_\Omega g~u~\dd x + \beta~ J(u),\\
\label{problem1}
\end{equation}
where we assume that $g \in L^p(\Omega)$ for some $p>1$. Obviously, \eqref{problem1} is non-convex due to the binary constraint on $u$. Convexity is obtained by considering the relaxation
\begin{equation}
\min\limits_{u\in BV(\Omega;[0,1])} \int_\Omega g~u~\dd x + \beta~ J(u).\\
\label{problem1relaxed}
\end{equation}
The existence of a minimizer to \eqref{problem1} and \eqref{problem1relaxed}, respectively, is addressed next.

\begin{proposition}\label{prop2.1}
Problem \eqref{problem1} and problem \eqref{problem1relaxed} admit minimizers, respectively.
\end{proposition}
\begin{proof}
For any $u \in BV(\Omega;S)$ with either $S=\{0,1\}$ or $S=[0,1]$ we have
$$ \int_\Omega g~u~\dd x \geq -\|g\|_{L^1(\Omega)}. $$          
Thus, on sublevel sets of 
the objective functionals in \eqref{problem1} and \eqref{problem1relaxed}, respectively, 
the total variation is uniformly bounded. Since the elements are additionally bounded in 
$L^\infty(\Omega) \hookrightarrow L^q(\Omega)$ for all $q < \infty$, we obtain the weak-* compactness
of level sets in $BV(\Omega;S) \cap L^q(\Omega)$. Based on the weak lower semicontinuity of $J$  according to \cite[Thm. 2.4]{MR1079985} (see also \cite[Thm. 1.9]{BVBook2} for $BV(\Omega)$), choosing $q=\frac{p}{p-1}$ we also obtain the weak lower semicontinuity of the objective functional and thus the existence of a minimizer 
in either case.
\end{proof}

The main result of this section, Theorem \ref{mainthm1} below, states that the relaxation \eqref{problem1relaxed} of
\eqref{problem1} is exact, in the sense that minimizers of the original
problem \eqref{problem1} are obtained from minimizers of its relaxation. Theorem \ref{mainthm1} also provides a
constructive way to compute (almost surely) minimizers of \eqref{problem1} by taking level sets of minimizers
of \eqref{problem1relaxed}. This result extends a corresponding observation in \cite{analysisL1TV1}, where total variation based image denoising with an $L^1$ data fidelity and given image data corresponding to a characteristic function of a bounded domain is considered. For its proof we recall that the perimeter $\text{Per}(E)$ of a measurable set $E\subset\Omega$ relative to $\Omega$ satisfies $\text{Per}(E)=\int_\Omega |D \chi_E|$, where $\chi_E$ denotes the characteristic function of $E$ relative to $\Omega$.
\begin{theorem} \label{mainthm1}
Let $u \in BV(\Omega;[0,1])$ be a minimizer of \eqref{problem1relaxed}. Then, for 
almost every $t \in (0,1)$, the function $u^t \in BV(\Omega;\{0,1\})$ defined as 
the indicator function of the level set $\{ u > t\}$, i.e.,
$$ u^t(x) = \left\{ \begin{array}{ll} 1, & \text{if } u(x) > t, \\ 0, &\text{else,} \end{array} \right. $$
is a minimizer of \eqref{problem1} and \eqref{problem1relaxed}.
\end{theorem}
\begin{proof}
Let $u \in BV(\Omega;[0,1])$ be a minimizer of \eqref{problem1relaxed} and
$v \in BV(\Omega;\{0,1\})$ be a minimizer of \eqref{problem1}.
 Then, due to the coarea formula \cite[Thm. I]{coarea} and Fubini's theorem we have
\begin{eqnarray*}
 \int_\Omega g~u~\dd x +\beta~ J(u) &=& \int_\Omega \int_0^{u(x)} g(x)~\dd t~\dd x+
 \beta~\int_{0}^1 \text{Per}(\{ u > t\})~\dd t  \\
 &=& \int_\Omega \int_0^1 g(x)~ u^t(x)~~\dd t~\dd x + \beta~\int_{0}^1 J(u^t)~\dd t \\
 &=& \int_0^1 \left( \int_\Omega g~u^t~\dd x +\beta~ J(u^t) \right)~\dd t \\
 &\geq& \int_0^1 \left( \int_\Omega g~v~\dd x +\beta~ J(v) \right)~\dd t \\
 &=& \int_\Omega g~v~\dd x +\beta~ J(v).
\end{eqnarray*}
Here we have used the fact that $u^t \in BV(\Omega;\{0,1\})$ to obtain the above inequality. One immediately observes from the above estimate that
$$ \int_\Omega g~u~\dd x +\beta~ J(u)  > \int_\Omega g~v~\dd x + \beta~J(v) $$
if $u^t$ is not a minimizer of \eqref{problem1}, which contradicts the minimizing
property of $u$. Hence, we conclude the assertion. 
\end{proof}

Based on Theorem \ref{mainthm1}, algorithmically we may find a minimizer of the binary constrained non-convex problem \eqref{problem1} by solving its convex relaxation \eqref{problem1relaxed} followed by a thresholding step.  
Further note that since the objective of \eqref{problem1relaxed} is convex but not strictly convex, we cannot expect uniqueness of the solution to the minimization problem. However, in situations where the solution
is indeed unique, we conclude from Theorem \ref{mainthm1} that the minimizer of \eqref{problem1} is unique and is the same as for \eqref{problem1relaxed}.
\begin{corollary}
If the minimizer $u$ of \eqref{problem1relaxed} is unique, then $u \in BV(\Omega;\{0,1\})$ 
and $u$ is the unique minimizer of \eqref{problem1}.
\end{corollary}

\subsection{The full problem}

Next we return to the full problem \eqref{full_model}. The associated relaxed version reads as follows:
\begin{equation}\label{problem0b}
\min_{(u,v)\in BV(\Omega;[0,1])\times\mathcal{V}} H(u,v):=F(v) + \int_\Omega G(v)~u~\dd x +\beta~ J(u).
\end{equation}
For the existence of a minimizer to \eqref{full_model} and \eqref{problem0b}, respectively, we first discuss the sequential weak lower semicontinuity of the objective functional. 
\begin{lemma}
Suppose that $\{v^k\}$ converges weakly in $\mathcal{V}$ to $v\in\mathcal{V}$, and $\{u^k\}$ weak-* converges in $\text{BV}(\Omega;[0,1]) \cap L^\infty(\Omega)$ to $u\in\text{BV}(\Omega;[0,1]) \cap L^\infty(\Omega)$ as $k\to+\infty$. Then, it holds that
$$ F(v) + \int_\Omega G(v)u~\dd x +\beta~ J(u) \leq 
\liminf_{k\to+\infty} \left( F(v^k) + \int_\Omega G(v^k)u^k~\dd x + \beta~J(u^k)\right) . $$
\end{lemma}
\begin{proof}
Due to the strong continuity of $G$, we obtain that $G(v^k)$ strongly converges to $G(v)$ in $L^1(\Omega)$. Because of the weak-* convergence of $u^k$ in $L^\infty(\Omega)$, we have 
$$ \int_\Omega G(v^k)u^k~\dd x  \rightarrow \int_\Omega G(v)u~\dd x\quad\text{as }k\rightarrow+\infty.$$
Since $F$ is sequentially weakly lower semicontinuous on ${\cal V}$ and $J$ is lower 
semicontinuous in the strong topology of $L^1(\Omega)$, we obtain 
\begin{eqnarray*}
&& \liminf_{k\to+\infty} \left( F(v^k) + \int_\Omega G(v^k)u^k~\dd x + \beta~J(u^k)\right) \\
&& \qquad \qquad \geq  \liminf_{k\to+\infty} F(v^k) + \lim_{k\to+\infty} \int_\Omega G(v^k)u^k~\dd x + 
\beta~\liminf_{k\to+\infty} J(u^k) \\
&& \qquad \qquad \geq F(v) + \int_\Omega G(v)u~\dd x +\beta~ J(u),
\end{eqnarray*}
which completes the proof.
\end{proof}

Another important property in the analysis of the minimization problem is the compactness of level sets associated with the objective function.
\begin{lemma} \label{fullcompactnesslemma}
Assume that the set 
$$\{F(v) - \epsilon \norm{G(v)}_{L^p}^2~|~v\in {\cal V}\}$$ 
is precompact in the weak topology of ${\cal V}$ for some $\epsilon > 0$
or for $\epsilon = 0$ with the additional condition that 
$G(\cdot) \geq g_0$ for some $g_0 \in \R$.
Then the sub-level sets of the functional $H$ are precompact in the weak topology of ${\cal V}$
with respect to $v$ and in the weak-* topologies of $BV(\Omega)$ and $L^\infty(\Omega)$ with
respect to $u$, if $0 \leq u \leq 1$ almost everywhere (a.e.) in $\Omega$. 
\end{lemma}
\begin{proof}
We first study the case $\epsilon>0$.
Using H\"older's inequality, the fact that 
$$\norm{u}_{L^q}^q \leq \int_\Omega u~\dd x \leq \abs{\Omega},$$
which is due to $0\leq u\leq 1$ a.e. in $\Omega$, and Young's inequality, we estimate
$$ H(u,v) \geq F(v) - \norm{G(v)}_{L^p} \norm{u}_{L^q} + \beta~J(u) \geq 
F(v) - \epsilon \norm{G(v)}_{L^p}^2 - \frac{\abs{\Omega}^{2/q}}{4 \epsilon} +\beta~ J(u), $$
where $|\Omega|$ denotes the volume of the set $\Omega$ and $p^{-1}+q^{-1}=1$. Since $J(u)$ is nonnegative, we conclude that boundedness of $H(u,v)$ implies
boundedness of $F(v) - \epsilon \norm{G(v)}_{L^p}^2$. 

In the case where $G(v) \geq g_0$ 
we estimate 
$$ H(u,v) \geq F(v) + \int_\Omega \min\{0,g_0\} ~\dd x +\beta~ J(u) \geq F(v) - \abs{g_0} \abs{\Omega} + \beta~J(u), $$
and thus, boundedness of $H(u,v)$ implies boundedness of $F(v)$. Under the above conditions, 
we obtain in both cases precompactness with respect to $v$ in the weak topology of ${\cal V}$ on sublevel sets of 
$H$. Due to the strong continuity of $G$, $\norm{G(v)}_{L^1}$ is bounded by some constant $c_1>0$ on a weakly compact set in ${\cal V}$, and thus $\int_\Omega G(v) u~\dd x \geq - c_1 \abs{\Omega}$.
Together with the fact that $F$ is bounded from below, we obtain that $J$ is uniformly bounded 
on sublevel sets of $H$. Moreover, we have $\norm{u}_{L^\infty} \leq 1$, which implies 
$\norm{u}_{L^1} \leq \abs{\Omega}$, and hence, $\norm{u}_{BV}$ is uniformly bounded. As a direct consequence
we obtain the weak-* precompactness with respect to $u$. 
\end{proof}

Now we have verified the assumptions of the fundamental theorem of optimization for the functional 
$H$, and as a consequence we establish an existence result.
\begin{corollary}\label{cor.fund}
Under the conditions of Lemma \ref{fullcompactnesslemma} there exists a minimizer of the functional 
$H$ on $BV(\Omega;[0,1]) \times {\cal V}$ as well as on $BV(\Omega;\{0,1\}) \times {\cal V}$.
\end{corollary} 

The proof of Corollary \ref{cor.fund} utilizes the usual infimal sequence argument relying on weak-* respectively weak convergence of a subsequence of $\{u^k\}$ and $\{v^k\}$ and the weak lower semicontinuity result of Lemma \ref{fullcompactnesslemma} as well as \cite[Thm. 2.4]{MR1079985} and \cite[Thm. 1.9]{BVBook2}.

We finally derive the main result of this section, a relaxation property similarly to the one obtained in Theorem
\ref{mainthm1}.
\begin{theorem} \label{mainthm2}
Let the assumptions of Lemma \ref{fullcompactnesslemma} be fulfilled and let 
$(\ol u, \ol v) \in BV(\Omega;[0,1]) \times {\cal V}$ be a minimizer of \eqref{problem0b}. Then, for almost every $t \in (0,1)$, the pair $(u^t,\ol v) \in BV(\Omega;\{0,1\}) \times {\cal V}$, 
where $u^t$ is
the indicator function of the level set $\{ \ol u > t\}$, i.e.,
$$ u^t(x) = \left\{ \begin{array}{ll} 1, & \text{if } \ol u > t, \\ 0, &\text{else,} \end{array} \right.$$
is a minimizer of \eqref{full_model} and \eqref{problem0b}.
\end{theorem}
\begin{proof}
Since $(\ol u,\ol v) \in BV(\Omega;[0,1]) \times {\cal V}$ is a minimizer of \eqref{problem0b}, then in particular
$\ol u$ is a minimizer of 
$$  \min_{u \in BV(\Omega;[0,1]) }H(u,\ol v),$$
which is equivalent to the auxiliary problem \eqref{problem1relaxed} with $g:=G(\ol v)$. 
Hence, from Theorem \ref{mainthm1} we conclude that for almost every $t$ the indicator function $u^t$ is a minimizer, too.  From this, we infer
$$ H(u^t,\ol v) = H(\ol u, \ol v) =  \inf_{(u,v) \in BV(\Omega;[0,1]) \times {\cal V}} H(u,v)
\leq \inf_{(u,v) \in BV(\Omega;\{0,1\}) \times {\cal V}} H(u,v), $$
and, thus, $(u^t,\ol v)$ is a minimizer of \eqref{full_model} and \eqref{problem0b}.
\end{proof}

\subsection{Stability}

In the following we investigate the stability of \eqref{full_model}
with respect to perturbations of the operators. Since the stability with respect to 
$v$ can be obtained under standard conditions in Banach spaces \cite{funanalysis}, 
we restrict our attention to stability of the auxiliary problem \eqref{problem1} and its
relaxation \eqref{problem1relaxed} with respect to perturbations of $g$. For this purpose, we state an interesting 
preliminary result concerning the equivalence of  \eqref{problem1} to the optimization of an objective functional consisting of a quadratic term and the total variation functional.
\begin{lemma} \label{linearquadraticlemma}
For $g \in L^2(\Omega)$, \eqref{problem1} is equivalent to the regularized 
least-squares problem 
\begin{equation}
\min_{u \in BV(\Omega;\{0,1\})} \frac{1}{2} \int_\Omega (u+g- \frac{1}{2})^2\dd x +\beta~ J(u).
\label{problem1a}
\end{equation}
\end{lemma}
\begin{proof}
Let $x\in\Omega$. For $u(x) \in \{0,1\}$ we have $u(x)^2=u(x)$ and thus 
$$u(x)^2 + 2 u(x) \left(g(x) - \frac{1}{2}\right) = u(x) g(x). $$
Since this equality holds almost everywhere in $\Omega$ we deduce
$$ \int_\Omega g~u~\dd x= \frac{1}{2} \int_\Omega (u+g- \frac{1}{2})^2\dd x 
- \frac{1}{2} \int_\Omega (g- \frac{1}{2})^2\dd x. $$
Hence, the objective functionals in \eqref{problem1} and \eqref{problem1a} only differ by
a constant term and consequently their minimizers are equal. 
\end{proof} 

Using the results of \cite{AcVo94} for regularized least-squares problems of the form
\eqref{problem1a} in $BV(\Omega)$ together with the closedness of 
$BV(\Omega;\{0,1\})$ with respect to strong $L^p$-convergence (cf. e.g.
\cite{DeZo01}), we can immediately deduce a stability result for the geometric problem.
\begin{proposition} \label{geometricstability}
Let $g^k \rightarrow g$ strongly in $L^2(\Omega)$ and let $u^k \in BV(\Omega;\{0,1\})$ 
denote a minimizer of \eqref{problem1} with $g$ replaced by $g^k$. Then, 
$(u^k)_k$ has a strongly converging subsequence in $L^p(\Omega)$ for each $p \in [1,\infty)$,
and each limit of a converging subsequence is a minimizer of \eqref{problem1}.
\end{proposition}

As a direct consequence of Proposition \ref{geometricstability} and Theorem \ref{mainthm1}, 
we obtain the following stability result for the level sets of minimizers of the perturbed relaxed problem.
\begin{corollary} \label{stabilityresult}
Let $g^k \rightarrow g$ strongly in $L^2(\Omega)$ and let $u^k \in BV(\Omega;[0,1])$ 
denote a minimizer of \eqref{problem1relaxed} with $g$ replaced by $g^k$. Then, for almost 
every $t \in (0,1)$, the indicator function of $\{u^k > t\}$ has a strongly converging 
subsequence in $L^p(\Omega)$ for each $p \in [1,\infty)$,
and each limit of a converging subsequence is a minimizer of \eqref{problem1}.
\end{corollary}

\subsection{Quadratic penalization -- uniqueness}

Since the objective functionals in \eqref{full_model} and \eqref{problem1} are not strictly convex
with respect to $u$, as stated before we can expect uniqueness of solutions neither for these problems nor for their
relaxations. 

In order to obtain an approximate problem with a unique solution, it seems natural to add a
quadratic regularization term with a small parameter $\epsilon>0$. For this reason, we consider 
\begin{equation}
\min_{u \in BV(\Omega;[0,1])} \int_\Omega g~u~\dd x + \beta~J(u) + \frac{\epsilon}2 \int_\Omega (u - \hat u)^2 \dd x 
\label{problem1relaxeda}
\end{equation}
with a given $\hat{u}\in BV(\Omega;[0,1])$. We have the following existence and uniqueness result.
\begin{proposition}
For $\epsilon > 0$ the optimization problem \eqref{problem1relaxeda} has a unique solution $u^\epsilon\in BV(\Omega;[0,1])$. 
\end{proposition}
\begin{proof}
First of all, existence of a minimizer can be concluded by similar arguments as for
\eqref{problem1relaxed}. Since the first two terms of the objective functional are convex, and
the additional quadratic term is strictly convex, we conclude uniqueness of the minimizer.
\end{proof}

Due to the additional quadratic term, the relaxation property of Theorem \ref{mainthm1} does
not hold for \eqref{problem1relaxeda}. Consequently, the minimum in $BV(\Omega;\{0,1\})$ might be
strictly greater than the one in $BV(\Omega;[0,1])$. However, we are only interested in \eqref{problem1relaxeda}
as an approximation to problem \eqref{problem1} and therefore study the convergence of 
level sets as $\epsilon \rightarrow 0$.

\begin{theorem}
For every sequence $\epsilon_k \rightarrow 0$ there exists a subsequence (again denoted by
$(\epsilon_k)_k$) such that for almost every $t \in (0,1)$, the sequence $(u^{\epsilon_k,t})_k$, with $u^{\epsilon_k,t}$ defined as the indicator function of the level set $\{ u^{\epsilon_k} > t \}$,
converges to a solution of problem \eqref{problem1} in the strong topology of $L^p(\Omega)$ for 
each $p \in [1,\infty)$. Moreover, if for any sequence $\epsilon_k \rightarrow 0$, the sequence 
$u^{\epsilon_k,t}$ converges, then the limit is a solution of problem \eqref{problem1relaxed} .
\end{theorem}
\begin{proof}
Let $u^o$ be a minimizer of \eqref{problem1relaxed}. Based on the definition of $u^\epsilon$ and $u^o$, we conclude 
\begin{eqnarray*}
\int_\Omega g~u^\epsilon~\dd x + \beta~J(u^\epsilon) + \frac{\epsilon}2 \int_\Omega (u^\epsilon - \hat u)^2 \dd x &\leq& 
\int_\Omega g~u^o~\dd x + \beta~J(u^o) + \frac{\epsilon}2 \int_\Omega (u^o - \hat u)^2 \dd x \\
&\leq& \int_\Omega g~u^\epsilon~\dd x +\beta~ J(u^\epsilon) + \frac{\epsilon}2 \int_\Omega (u^o - \hat u)^2 \dd x ,
\end{eqnarray*}
and thus, 
$$ \int_\Omega (u^\epsilon - \hat u)^2 \dd x \leq \int_\Omega (u^o - \hat u)^2 \dd x ,$$
which implies in particular that $\norm{u^\epsilon - \hat u}_{L^2}$ is uniformly bounded in $\epsilon$.
Let $g^\epsilon:=g + \epsilon(u^\epsilon - \hat u)$, then $u^\epsilon$ is also a minimizer of
$$\min_{u \in BV(\Omega)}  \int_\Omega g^\epsilon~u~\dd x +\beta~ J(u)\quad\text{ subject to (s.t)}\quad u(x)\in[0,1],$$
which can be seen from the associated necessary and sufficient optimality condition.  Since 
$$ \norm{g^\epsilon - g}_{L^2} = \epsilon \norm{u^\epsilon - \hat u}_{L^2} = {\cal O}(\epsilon), $$
we can employ Corollary \ref{stabilityresult} to conclude the assertion. 
\end{proof}


\subsection{Volume constraints} 
In many applications, volume constraints on $u$ are relevant. A particular instance occurs in topology optimization, where such a constraint fixes the amount of material available for designing an optimal topological structure; see section~\ref{ssec:mincompliance} below for an application in minimal compliance topology optimization. 

In view of this discussion, the model \eqref{full_model} is augmented by adding the volume constraint
\begin{equation}
\int_\Omega u~\dd x = V,
\label{volumeconstraint}
\end{equation}
with fixed $0  < V < \abs{\Omega}$.

We observe that the range space of this constraint is finite (actually one-) dimensional. As the following example highlights, constraints in function space could be problematic.

\begin{example}
Let ${\cal V} = L^2(\Omega)$, $F(v) = \int_\Omega \abs{v-\frac{1}2}^2~dx$, 
and $G(v) \equiv 0$. We consider the relaxed problem \eqref{problem0b} subject to the constraint
$v=u$ in $L^2(\Omega)$. Obviously, the objective functional vanishes for $v = u \equiv \frac{1}2$, 
and it is positive for any other $(u,v)  \in BV(\Omega;[0,1]) \times L^2(\Omega)$. 
Thus, $v = u \equiv \frac{1}2$ is the unique minimizer, and in particular, the infimum over 
$BV(\Omega;\{0,1\})$ is strictly larger than $0$. 
\end{example}

Possibly the main reason for the fact that the ''exact'' relaxation result does not hold for general constraints is a potential
nonexistence of Lagrange multipliers for the binary-valued problem. In the case of the volume constraint
\eqref{volumeconstraint}, however, we can verify the existence of a (scalar) Lagrange multiplier
and use it to prove an equivalence result for the relaxation in this case. We start again with the minimization
of the auxiliary problem from above, but now subject to \eqref{volumeconstraint}.
\begin{theorem} \label{volumethm1}
Let $u \in BV(\Omega;[0,1])$ be a minimizer of \eqref{problem1relaxed} subject to 
\eqref{volumeconstraint}. Then, for 
almost every $t \in (0,1)$, the function $u^t \in BV(\Omega;\{0,1\})$ defined as 
the indicator function of the level set $\{ u > t\}$, i.e.,
$$ u^t(x) = \left\{ \begin{array}{ll} 1, & \text{if } u > t, \\ 0, &\text{else,} \end{array} \right. $$
is a minimizer of \eqref{problem1} and \eqref{problem1relaxed} subject to 
\eqref{volumeconstraint}, respectively.
\end{theorem}
\begin{proof}
In the following let $u^\lambda$ (for $\lambda \in \R$) denote a minimizer of the problem 
\begin{equation}\label{volumemodel}
\min_{u \in BV(\Omega;[0,1])} \int_\Omega (g+\lambda)~u~\dd x + \beta~J(u).
\end{equation}
Existence of $u^\lambda$ follows from Proposition~\ref{prop2.1}.
Moreover, we can apply the above relaxation results and therefore we may assume that 
$u^\lambda \in BV(\Omega;\{0,1\})$. By comparing objective functionals, it is easy to show that we have for $\epsilon > 0$
$$ \int_\Omega u^{\lambda+\epsilon}~\dd x \leq \int_\Omega u^{\lambda}~\dd x. $$
Further, using the weak-* compactness in $L^\infty(\Omega) \cap BV(\Omega;\{0,1\})$ we can prove for 
$\mu_k \rightarrow \lambda$ that a subsequence $(u^{\mu_k})_k$ converges to $u^\lambda$ as $k\to+\infty$ in the weak-* topology of $L^\infty(\Omega)$ 
and, thus, in particular it holds that
$$  \int_\Omega u^{\mu_k}~\dd x \rightarrow \int_\Omega u^{\lambda}~\dd x. $$
Hence, we conclude by standard arguments for real functions that the map
$$ W:\R \rightarrow \R, \quad \lambda \mapsto \int_\Omega u^\lambda~\dd x$$
is well-defined, continuous and non-increasing. As $\lambda \rightarrow \infty$ we must have 
$\int_\Omega u^\lambda~\dd x \rightarrow 0$ due to uniform boundedness of the objective functional, and
for $\lambda \rightarrow - \infty$ we obtain $\int_\Omega u^\lambda~\dd x \rightarrow \abs{\Omega}$.
Since $W$ is a continuous function on $\R$, we conclude that for each $V \in (0,\abs{\Omega})$ there exists
$\lambda \in \R$ such that $W(\lambda) = V$. Moreover, for all $u \in BV(\Omega;[0,1])$ with 
$\int_\Omega u~\dd x= V$ we have from the definition of $u^\lambda$
\begin{eqnarray*}
\int_\Omega g~u^\lambda~\dd x  + \beta~J(u^\lambda) &=& 
\int_\Omega (g+\lambda)~u^\lambda~\dd x + \beta~J(u^\lambda) - \lambda V \\
&\leq& \int_\Omega (g+\lambda)~u~\dd x + \beta~J(u) - \lambda V \\
 &=& \int_\Omega g~u~\dd x +\beta~ J(u).
 \end{eqnarray*}
 Thus, $u^\lambda$ is a solution of the relaxed problem \eqref{problem1relaxed} subject to 
\eqref{volumeconstraint}. 

Now take any solution $\ol u \in BV(\Omega;[0,1])$ of \eqref{problem1relaxed} subject to 
\eqref{volumeconstraint}. Then again it holds that
$$ \int_\Omega (g+\lambda)~ u^\lambda~\dd x + \beta~J(u^\lambda) = \int_\Omega (g+\lambda)~\ol u~\dd x + \beta~J(\ol u)$$
and hence, $\ol u$ is also a solution of \eqref{volumemodel}.
Now we can apply Theorem \ref{mainthm1} to this problem and show that the indicator functions $u^t$ 
are solutions, too, for almost every $t$. Then, by the above argument, $u^t$ is also a solution of 
\eqref{problem1relaxed} subject to \eqref{volumeconstraint}.
\end{proof}

By reasoning analogous to the one in the proof of Theorem \ref{mainthm2} one shows the following result.
\begin{theorem} 
Let the assumptions of Theorem \ref{mainthm2} be fulfilled and let 
$(\ol u, \ol v) \in BV(\Omega;[0,1]) \times {\cal V}$ be a minimizer of 
\eqref{problem0b} subject to \eqref{volumeconstraint}.
Then, for almost every $t \in (0,1)$, the pair $(u^t,\ol v) \in BV(\Omega;\{0,1\}) \times {\cal V}$, 
where $u^t$ is
the indicator function of the level set $\{ \ol u > t\}$, i.e.,
$$ u^t(x) = \left\{ \begin{array}{ll} 1, & \text{if } \ol u > t, \\ 0, &\text{else,} \end{array} \right.$$
is a minimizer of \eqref{full_model} and \eqref{problem0b} subject to \eqref{volumeconstraint}, respectively.
\end{theorem}

\section{The Fenchel dual problem and dual semismooth Newton}

This section is devoted to the study of an efficient algorithm for solving the convex relaxed auxiliary problem \eqref{problem1relaxed} with $g\in L^p(\Omega)$ for some $p>1$. 

In order to remove the box constraint in \eqref{problem1relaxed}, we add a penalty term to obtain an unconstrained minimization problem of the form
\begin{equation}\label{unmin1}
\min\limits_{u\in BV(\Omega)} \int_\Omega g~u~\dd x + c\int_\Omega \max\left(0,2\left|u-\frac12\right|-1\right)~\dd x + \beta~J(u)
\end{equation}
with a penalty parameter $c>0$. Note that the max-term penalizes violations of $u\in[0,1]$ in \eqref{problem1relaxed} and yields an exact penalty \cite{exactpnlty1,exactpnlty2}, i.e., there exists $c^*>0$ such that the solution to \eqref{problem1relaxed} solves \eqref{unmin1}, as well. 

In order to have uniqueness of a solution, which is not guaranteed for \eqref{unmin1}, we consider next
\begin{align}\label{unmin2}
\min\limits_{u\in BV(\Omega)}& \int_\Omega g~u~\dd x + c\int_\Omega \max\left(0,2\left|u-\frac12\right|-1\right)~\dd x\nonumber\\ &\qquad +\frac{\epsilon}{2} \int_\Omega |u-1|^2~\dd x +\frac{\epsilon}{2} \int_\Omega |u|^2~\dd x +\beta~J(u),\tag{$P$}
\end{align}
with arbitrary, but fixed $\epsilon>0$. Note that the additional quadratic terms (depending on the magnitude of $\epsilon$) prefer solutions with $u(x)\in\{0,1\}$.

In view of the non-differentiability of the TV regularization and the penalty term, respectively, we now study a dual formulation. For this purpose we formally compute the Fenchel dual of \eqref{unmin2} and set $\Lambda:=\nabla$ as well as
\begin{equation}
\mathcal{F}(u):=\int_\Omega g~u~\dd x + c\int_\Omega \max\left(0,2\left|u-\frac12\right|-1\right)~\dd x+\frac{\epsilon}{2} \int_\Omega |u-1|^2~\dd x +\frac{\epsilon}{2} \int_\Omega |u|^2~\dd x,\label{fun_F}
\end{equation}
and further
$$
\mathcal{G}(\vec{p}):=\beta\int_\Omega |\vec{p}|_1~\dd x,
$$
where $|\cdot|_1$ denotes the $l_1$-vector norm. Based on the definition of the Fenchel conjugate \cite{Analysis}, we get
\begin{align*}
\mathcal{F}^*(u^*)=&\frac{1}{4\epsilon}\big[\|\min(u^*-g+\epsilon+2c,0)\|^2+\|\max(u^*-g+\epsilon,0)\|^2
\\&+\|\max(u^*-g-\epsilon-2c,0)\|^2-\|\max(u^*-g-\epsilon,0)\|^2\big]-\frac{\epsilon}{2}|\Omega|,\\
\mathcal{G}^*(\vec{p}^*)=&I_{[-\beta\vec{1},\beta\vec{1}]}(\vec{p}^*),
\end{align*}
where 
\[
I_{[-\beta\vec{1},\beta\vec{1}]}(\vec{p}^*)=\left\{\begin{array}{ll}0,& \mbox{if}\ -\beta\vec{1}\leq\vec{p}^*(x)\leq\beta\vec{1}\ \mbox{for a.e.}\ x\in\Omega,\\ \infty, & \mbox{otherwise.}                                             
\end{array}\right.
\]
Thus, formally the dual of \eqref{unmin2} is given by
\begin{align}
\min\limits_{|\vec{q}(x)|_\infty\leq\beta \mbox{\scriptsize\ a.e. in } \Omega}&\quad \frac{1}{4\epsilon} \big[\|\min(\mathrm{div}\vec{q}-g+\epsilon+2c,0)\|^2+\|\max(\mathrm{div}\vec{q}-g+\epsilon,0)\|^2\nonumber\\
&\hspace{-2cm}+\|\max(\mathrm{div}\vec{q}-g-\epsilon-2c,0)\|^2-\|\max(\mathrm{div}\vec{q}-g-\epsilon,0)\|^2\big]-\frac{\epsilon}{2}|\Omega|,\hspace{1cm} \tag{$P^*$}\label{dual0}
\end{align}
where $\vec{q}=-\vec{p}^*$.

Based on first-order optimality conditions, solutions $\bar{u}$ and $\bar{\vec{q}}$ of \eqref{unmin2} and \eqref{dual0}, respectively, satisfy
\begin{align}
\bar{u}=&\left\{\begin{array}{ll}
           \frac{1}{2\epsilon}(\mathrm{div}\bar{\vec{q}}-g+\epsilon+2c), & \mathrm{if}\ \mathrm{div}\bar{\vec{q}}-g<-\epsilon-2c, \\
           0, & \mathrm{if}\ \mathrm{div}\bar{\vec{q}}-g\in[-\epsilon-2c,-\epsilon], \\
           \frac{1}{2\epsilon}(\mathrm{div}\bar{\vec{q}}-g+\epsilon), & \mathrm{if}\ |\mathrm{div}\bar{\vec{q}}-g|<\epsilon, \\
           1, & \mathrm{if}\ \mathrm{div}\bar{\vec{q}}-g\in[\epsilon,\epsilon+2c], \\
           \frac{1}{2\epsilon}(\mathrm{div}\bar{\vec{q}}-g+\epsilon-2c), & \mathrm{if}\ \mathrm{div}\bar{\vec{q}}-g>\epsilon+2c, \\
         \end{array}\right.\label{eqn_u}\\
\bar{\vec{q}}=&\beta\bigg(\frac{\bar{u}_{x_i}}{|\bar{u}_{x_i}|}\bigg)^n_{i=1}\ \mathrm{on}\ \{x: \bar{u}_{x_i}(x)\neq0\mbox{ for all }i\}.
\end{align}
We observe that if $\epsilon$ is sufficiently small and $c$ is sufficiently large such that $|\mathrm{div}\bar{\vec{q}}-g|$ is always in $[\epsilon,\epsilon+2c]$, then the expected solution $\bar{u}\in\{0,1\}$ is found directly. However, for $x\in\Omega$ with $\bar{u}(x)\in\{0,1\}$, we only get the characterization $\epsilon\leq|\mathrm{div}\bar{\vec{q}}(x)-g(x)|\leq\epsilon+2c$, but we do not obtain exact values for $\mathrm{div}\bar{\vec{q}}(x)-g(x)$. In order to cope with this incomplete characterization of $\mathrm{div}\bar{\vec{q}}-g$, we propose to add a regularization term to the objective of \eqref{dual0} yielding
\begin{align}\label{dual1}
\min\limits_{|\vec{q}(x)|_\infty\leq\beta \mbox{\scriptsize\ a.e. in } \Omega} &\quad \frac{\gamma}{2}\|\mathrm{div}\vec{q}-g\|^2_2+\frac{1}{4\epsilon}\big[\|\min(\mathrm{div}\vec{q}-g+\epsilon+2c,0)\|^2_2\nonumber\\
&+\|\max(\mathrm{div}\vec{q}-g+\epsilon,0)\|_2^2+\|\max(\mathrm{div}\vec{q}-g-\epsilon-2c,0)\|_2^2\nonumber\\
&-\|\max(\mathrm{div}\vec{q}-g-\epsilon,0)\|_2^2\big]-\frac{\epsilon}{2}|\Omega|\tag{$P^*_\gamma$},
\end{align}
where $\gamma>0$ is the dual regularization parameter. To understand the structural impact of the $\gamma$-regularization, we study the dual of \eqref{dual1} and relate the resulting problem to \eqref{unmin2}. For this purpose we introduce the function 
\begin{align*}
\mathcal{F}^{*}_\gamma(u^*)=&\frac{\gamma}{2}\|u^*-g\|^2_2+\frac{1}{4\epsilon}\big[\|\min(u^*-g+\epsilon+2c,0)\|^2_2+\|\max(u^*-g+\epsilon,0)\|_2^2
\\&+\|\max(u^*-g-\epsilon-2c,0)\|_2^2-\|\max(u^*-g-\epsilon,0)\|_2^2\big]-\frac{\epsilon}{2}|\Omega|.
\end{align*}
Based on the definition of the Fenchel conjugate we obtain the dual of $\mathcal{F}^*_\gamma$ as 
\begin{align}\label{fun_F**}
\mathcal{F}^{**}_\gamma(u^{**})=&\int_\Omega u^{**}~g~\dd x- \frac{1}{2\gamma(1+2\epsilon\gamma)}\big[\|\min(u^{**}+(\epsilon+2c)\gamma,0)\|^2_2\nonumber\\
&+\|\max(u^{**}+\epsilon\gamma,0)\|_2^2+\|\max(u^{**}-1-(\epsilon+2c)\gamma,0)\|_2^2\nonumber\\
&-\|\max(u^{**}-1-\epsilon\gamma,0)\|_2^2\big]+\frac{1}{2\gamma}\|u^{**}\|^2+\frac{\epsilon}{2}|\Omega|.
\end{align}
Comparing $\mathcal{F}^{**}_\gamma$ with $\mathcal{F}$ in the original problem \eqref{unmin2}, we come to the following conclusion.

\begin{proposition}
As $\gamma$ tends to zero, $\mathcal{F}^{**}_\gamma$ defined in \eqref{fun_F**} converges to $\mathcal{F}$ defined in \eqref{fun_F}.
\end{proposition}
\begin{proof}
For making the relation between $\mathcal{F}^{**}_\gamma$ and $\mathcal{F}$ more explicit, we define
\begin{align*}
\Omega_1=&\{x\in\Omega: u(x)\leq-(\epsilon+2c)\gamma\}, & \Omega_2=&\{x\in\Omega: u(x)\in(-(\epsilon+2c)\gamma,-\epsilon\gamma)\},\\
\Omega_3=&\{x\in\Omega: u(x)\in[-\epsilon\gamma,1+\epsilon\gamma]\}, & \Omega_4=&\{x\in\Omega: u(x)\in(1+\epsilon\gamma,1+(\epsilon+2c)\gamma)\},\\
\Omega_5=&\{x\in\Omega: u(x)\geq1+(\epsilon+2c)\gamma\}.&
\end{align*}
Obviously, we have $\Omega=\bigcup_{i=1}^5\Omega_i$, and $\Omega_i\bigcap\Omega_j=\emptyset$ for any $i\neq j$ with $i,j=1,\cdots,5$. Then, $\mathcal{F}^{**}_\gamma$ can be equivalently written as
\begin{align*}
\mathcal{F}^{**}_\gamma(u)=&\int_\Omega u~g~\dd x+ \frac{c}{1+2\epsilon\gamma}\int_\Omega \max(0,2|u-\frac12|-1-2\epsilon\gamma)~\dd x\\& +\frac{1}{1+2\epsilon\gamma}\bigg[\frac{\epsilon}{2} \int_\Omega |u-1|^2~\dd x +\frac{\epsilon}{2} \int_\Omega |u|^2~\dd x\bigg]-\frac{\gamma(\epsilon^2+4c^2)}{2(1+2\epsilon\gamma)}|\Omega_1|\\
&+\frac{1}{1+2\epsilon\gamma}\int_{\Omega_2}\big[
\frac{1}{2\gamma}u^2+(\epsilon+2c)(u+\epsilon\gamma)\big]~\dd x-\frac{\gamma\epsilon^2}{2(1+2\epsilon\gamma)}|\Omega_3|\\&+\frac{1}{1+2\epsilon\gamma}\int_{\Omega_4}\big[
\frac{1}{2\gamma}(u-1)^2+(\epsilon+2c)(1+\epsilon\gamma-u)\big]~\dd x-\frac{\gamma(\epsilon^2+4c^2)}{2(1+2\epsilon\gamma)}|\Omega_5|.
\end{align*}
Based on the definitions of $\Omega_2$ and $\Omega_4$, we have 
\begin{align*}
\frac{\epsilon^2\gamma}{2}<\frac{1}{2\gamma}u^2<\frac{(\epsilon+2c)^2\gamma}{2}, &\qquad \ \forall x\in\Omega_2,\\
\frac{\epsilon^2\gamma}{2}<\frac{1}{2\gamma}(u-1)^2<\frac{(\epsilon+2c)^2\gamma}{2}, &\qquad \ \forall x\in\Omega_4.
\end{align*}
Hence, when $\gamma$ goes to zero, these two terms in $u$ are uniformly bounded. Furthermore, 
the integrals on $\Omega_2$ and $\Omega_4$ in the above definition of $\mathcal{F}_\gamma^{**}$ as well as the terms involving $|\Omega_i|$, $i\in\{1,3,5\}$, tend to 0. Thus, we obtain
$$\lim_{\gamma\rightarrow0}\mathcal{F}^{**}_\gamma(u)=\mathcal{F}(u).$$
This concludes the proof.
\end{proof}

Now we return to problem \eqref{dual1} with $\vec{q}\in H_0(\mathrm{div}):=\{\vec{v}\in\mathbf{L}^2(\Omega):\mathrm{div}\vec{v}\in L^2(\Omega),\vec{v}\cdot\vec{n}=0\ \mbox{on}\ \partial\Omega\}$ where $\vec{n}$ is the outward unit normal to $\partial\Omega$. We equip $H_0(\mathrm{div})$ with the norm $\|\vec{v}\|^2_{H_0(\mathrm{div})}=\|\vec{v}\|^2_{\mathbf{L}^2}+\|\mathrm{div}\vec{v}\|^2_2$. Then, standard arguments prove that \eqref{dual1} admits a solution in $H_0(\mathrm{div})$.

Since the divergence operator has a nontrivial kernel, the solution of \eqref{dual1} need not be unique. The following problem, on the other hand, has a unique solution:
\begin{align}\label{dual1_2}
\min\limits_{\scriptsize\begin{array}{c}\vec{q}\in H_0(\mathrm{div})\\
|\vec{q}(x)|_\infty\leq\beta \mbox{\scriptsize\ a.e. in } \Omega\end{array}} & \frac{\gamma}{2}\|\mathrm{div}\vec{q}-g\|^2_2+\frac{1}{4\epsilon}\big[\|\min(\mathrm{div}\vec{q}-g+\epsilon+2c,0)\|^2_2\nonumber\\
&+\|\max(\mathrm{div}\vec{q}-g+\epsilon,0)\|_2^2+\|\max(\mathrm{div}\vec{q}-g-\epsilon-2c,0)\|_2^2\nonumber\\
&-\|\max(\mathrm{div}\vec{q}-g-\epsilon,0)\|_2^2\big]+\frac{\lambda}{2}\|\mathrm{P}_{\mathrm{div}}\vec{q}\|_{\mathbf{L}^2}^2,
\end{align}
Above, $\mathrm{P}_{\mathrm{div}}$ denotes the orthogonal projection in $\mathbf{L}^2(\Omega)$ onto $H_0(\mathrm{div}0):=\{\vec{v}\in H_0(\mathrm{div}): \mathrm{div}\vec{v}=0\  \mbox{a.e. in}\ \Omega\}$. 

\begin{proposition}
The solution of \eqref{dual1_2} is unique.
\end{proposition}
\begin{proof}
Since the set of feasible $\vec{q}$ in \eqref{dual1_2} is convex and the term $\|\min(\mathrm{div}\vec{q}-g+\epsilon+2c,0)\|^2_2+\|\max(\mathrm{div}\vec{q}-g+\epsilon,0)\|_2^2+\|\max(\mathrm{div}\vec{q}-g-\epsilon-2c,0)\|_2^2-\|\max(\mathrm{div}\vec{q}-g-\epsilon,0)\|_2^2$ is convex but not strictly convex, the main target now is to show that $Q(\vec{q}):=\frac{\gamma}{2}\|\mathrm{div}\vec{q}-g\|^2_2+\frac{\lambda}{2}\|\mathrm{P}_{\mathrm{div}}\vec{q}\|_{\mathbf{L}^2}^2$ is strictly convex. Calculating the second derivative of $Q$, we get
\[
Q''(\vec{q},\vec{q})=\gamma\|\mathrm{div}\vec{q}\|^2_2+\lambda\|\mathrm{P}_{\mathrm{div}}\vec{q}\|_{\mathbf{L}^2}^2\geq\kappa\|\vec{q}\|^2_{H_0(\mathrm{div})}
\]
for a constant $\kappa>0$ independent of $\vec{q}\in H_0(\mathrm{div})$. From this we infer the assertion.
\end{proof}

\begin{remark}
In our numerics we frequently set $\lambda=0$ without running into difficulties due to non-uniqueness. This suggests that the box constraint $|\vec{q}(x)|_\infty\leq\beta$ often yields uniqueness.
\end{remark}

For solving \eqref{dual1_2} and in view of \cite{semismooth} (see also \cite{MR2219285}) we consider the approximating unconstrained minimization problem
\begin{align}\label{dual2}
\min\limits_{\vec{q}\in\mathbf{H}^1_0(\Omega)} & \frac{\gamma}{2}\|\mathrm{div}\vec{q}-g\|^2_2+\frac12\big[\|\min(\mathrm{div}\vec{q}-g+\epsilon+2c,0)\|^2_2+\|\max(\mathrm{div}\vec{q}-g+\epsilon,0)\|_2^2
\nonumber\\
&+\|\max(\mathrm{div}\vec{q}-g-\epsilon-2c,0)\|_2^2-\|\max(\mathrm{div}\vec{q}-g-\epsilon,0)\|_2^2\big]
+\frac{\lambda}{2}\|\mathrm{P}_{\mathrm{div}}\vec{q}\|_{\mathbf{L}^2}^2\nonumber\\
&+\frac{1}{2\alpha}\|\nabla \vec{q}\|_{\mathbf{L}^2}^2+\frac{\alpha}{2}\|\max(0,(\vec{q}-\beta))\|^2_{\mathbf{L}^2}+\frac{\alpha}{2}\|\min(0,(\vec{q}+\beta))\|^2_{\mathbf{L}^2},\tag{$P^*_{\gamma,\alpha}$}
\end{align}
where $\alpha>0$. Note that \eqref{dual2} and \eqref{dual1_2} differ in that we add three quadratic terms weighted by $\alpha>0$ in \eqref{dual2}. Furthermore, we keep here the same notations for $\gamma$ and $\lambda$ although these quantities equal the corresponding ones in \eqref{dual1_2}, but in \eqref{dual2} they are multiplied by $2\epsilon$, respectively. Straightforward arguments yield that the minimization problem \eqref{dual2} tends to \eqref{dual1} as $\alpha\rightarrow\infty$. The unique solution $\bar{\vec{q}}$ of \eqref{dual2} satisfies the first-order optimality condition
\begin{eqnarray}\label{ELeqn}
\mathcal{Q}(\bar{\vec{q}})&:=&-\nabla\big[\gamma(\mathrm{div}\bar{\vec{q}}-g)+\min(\mathrm{div}\bar{\vec{q}}-g+\epsilon+2c,0)+\max(\mathrm{div}\bar{\vec{q}}-g+\epsilon,0)\nonumber\\
&&\quad+\max(\mathrm{div}\bar{\vec{q}}-g-\epsilon-2c,0)-\max(\mathrm{div}\bar{\vec{q}}-g-\epsilon,0)\big]+\lambda\mathrm{P}_{\mathrm{div}}\bar{\vec{q}}\nonumber\\
&&\qquad-\frac{1}{\alpha}\vec{\Delta}\bar{\vec{q}}+\max(0,\alpha(\bar{\vec{q}}-\beta))
+\min(0,\alpha(\bar{\vec{q}}+\beta))=0.
\end{eqnarray}

It is well-known that the max- and min-operators are semismooth in our function space setting. This function space property of $\mathcal{Q}(\cdot)$ was the sole reason for adding $(2\alpha)^{-1}\|\nabla\vec{q}\|^2_{\mathbf{L}^2}$ in \eqref{dual2}. In a purely finite dimensional approach one may discard this term. As a consequence, for the efficient numerical solution of \eqref{ELeqn}, we utilize the semismooth Newton technique of \cite{semismooth2}. 

In our numerics we use the following generalized derivatives of the max- and the min-operations above:
$$
 D\max(\varphi,0)(x)=\left\{\begin{array}{ll}1,& \ \mbox{if}\ \varphi(x)>0,\\0, & \ \mbox{otherwise,}\end{array}\right.\quad\text{and}\quad
D\min(\varphi,0)(x)=\left\{\begin{array}{ll}1,& \ \mbox{if}\ \varphi(x)<0,\\0, & \ \mbox{otherwise.}\end{array}\right.
$$
%
Then the generalized Newton step $\delta\vec{q}_k\in\mathbf{H}^1_0(\Omega)$ at the current iterate $\vec{q}_k\in\mathbf{H}^1_0(\Omega)$ is the unique solution to
\[
\bigg[-\nabla\big(\gamma I+\chi_{\mathcal{A}^k_1}+\chi_{\mathcal{A}^k_2}+\chi_{\mathcal{A}^k_3}-\chi_{\mathcal{A}^k_4}
\big)\mathrm{div}+\lambda\mathrm{P}_{\mathrm{div}}-\frac{1}{\alpha}\vec{\Delta}
+\alpha(\chi_{\mathcal{A}^k_5}+\chi_{\mathcal{A}^k_6})\bigg]\delta\vec{q}=-\mathcal{Q}(\vec{q}_k),
\]
with the sets
\begin{align*}
\mathcal{A}^k_1=&\{\mathrm{div}\vec{q}_k-g<-\epsilon-2c\}, & \mathcal{A}^k_2=&\{\mathrm{div}\vec{q}_k-g>-\epsilon\},\\
\mathcal{A}^k_3=&\{\mathrm{div}\vec{q}_k-g>\epsilon+2c\}, & \mathcal{A}^k_4=&\{\mathrm{div}\vec{q}_k-g>\epsilon\},\\
\mathcal{A}^k_5=&\{\vec{q}_k>\beta\}, & \mathcal{A}^k_6=&\{\vec{q}_k<-\beta\}.
\end{align*}
The next iterate is defined as $\vec{q}^{k+1}:=\vec{q}^{k}+\delta\vec{q}^{k}$.
Using the results of \cite{semismooth2} it can be shown that the above iteration converges at a superlinear rate to the solution $\bar{\vec{q}}$ of \eqref{dual2} provided that the initial point $\vec{q}^0$ is sufficiently close to $\bar{\vec{q}}$.
We mention that the latter locality requirement of the initial point seemed not restrictive in our numerical practice; see below. 

Upon discretization, keeping the same notations as above, we have $g\in\mathbb{R}^n$, $\vec{q}^{k}\in\mathbb{R}^{dn}$ and, for $\mathcal{A}\subset\{1,\cdots,m\}$, $\chi_{\mathcal{A}}\in\mathbb{R}^{m\times m}$ denotes a diagonal matrix with the diagonal entries $t_i=1$ if $i\in\mathcal{A}$ and $t_i=0$ else for $i=1,\cdots,m$. Also, we have
\begin{align*}
\mathcal{A}^k_1=&\{i: (\mathrm{div}\vec{q}_k)_i-g_i<-\epsilon-2c\}, & \mathcal{A}^k_2=&\{i: (\mathrm{div}\vec{q}_k)_i-g_i>-\epsilon\},\\
\mathcal{A}^k_3=&\{i: (\mathrm{div}\vec{q}_k)_i-g_i>\epsilon+2c\}, & \mathcal{A}^k_4=&\{i: (\mathrm{div}\vec{q}_k)_i-g_i>\epsilon\},\\
\mathcal{A}^k_5=&\{i: (\vec{q}_k)_i>\beta\}, & \mathcal{A}^k_6=&\{i: (\vec{q}_k)_i<-\beta\}.
\end{align*}
Further, $I\in\mathbb{R}^{dn\times dn}$ denotes the identity matrix, $\mathrm{div}=-\nabla^{\top}\in\mathbb{R}^{n\times dn}$, and
\begin{align*}
\vec{\Delta}&=\left[\begin{array}{ccc}
               \Delta&&0\\&\ddots&\\0&&\Delta
              \end{array}\right]\in\mathbb{R}^{dn\times dn},\qquad \Delta\in\mathbb{R}^{n\times n},\\
{\nabla}&=\left[\begin{array}{c}
               \nabla_{x_1}\\\vdots\\ \nabla_{x_d}
              \end{array}\right]\in\mathbb{R}^{dn\times n},\qquad \nabla_{x_i}\in\mathbb{R}^{n\times n}\mbox{ with }i=1,\cdots,d.
\end{align*}
In our numerics, we discretize the Laplace operator by the standard five-point stencil with homogenous Dirichlet
boundary conditions and use forward differences for the gradient and corresponding discrete adjoint scheme for the divergence operator.

Notice that the system matrix of the semismooth Newton step is symmetric and positive definite. Consequently, we use the preconditioned conjugate gradient (PCG) method to solve the associated linear system. Upon convergence of the discrete Newton iteration, we obtain the solution $u$ to the discrete counterpart of \eqref{unmin2} based on the discrete version of \eqref{eqn_u}.

\section{Applications}

In the following we discuss several important applications of optimization problems with shapes as the unknown and apply the above relaxation results. Moreover, associated numerical results allow us to study the behavior of our semismooth Newton algorithm.

\subsection{The Rudin-Osher-Fatemi model for binary image denoising}
An important model for restoring images from data corrupted by Gaussian noise is the Rudin-Osher-Fatemi (ROF) problem \cite{TV1, TVmodel}
\begin{equation}\label{ROF}
\min_{u\in BV(\Omega)} \frac{1}{2}\int_\Omega (u-f)^2~\dd x+\beta~J(u) 
\end{equation}
where $\beta>0$ is a suitably chosen regularization parameter. If one is interested in binary images (e.g., black and white images), then the minimization process is constrained to $u\in BV(\Omega;\{0,1\})$. Clearly, the resulting problem is non-convex. However, as we shall see below, by using the techniques above there are several ways of relaxing \eqref{ROF} for binary images to obtain a convex problem.

Chan and Esedoglu \cite{analysisL1TV1} considered a version of the ROF-functional for binary images, where they replaced the quadratic fidelity term by the $L^1$-functional $\int_\Omega |u-f|\ dx$. Clearly, if $u$ and $f$ are binary, i.e., they take values in $\{0,1\}$ only, then $|u(x)-f(x)|\in\{0,1\}$ for almost every $x$ and hence, the $L^1$-functional is equivalent to the $L^2$-functional, since $|u-f|=|u-f|^2$. Using a proof technique similar to the one above, they could show that the relaxation of the L$^1$-functional to $BV(\Omega;[0,1])$ is equivalent in the sense that indicator functions of almost all level sets of minimizers are solutions of the problem in $BV(\Omega;\{0,1\})$. A disadvantage of this relaxation is the additional non-smooth term, which increases the computational effort in minimizing the relaxed functional.

Instead, we consider here a linear relaxation of the first term, using the equivalence result of Lemma \ref{linearquadraticlemma}. For this purpose we start with $u\in BV(\Omega;\{0,1\})$ and obtain
\[
\frac12\int_\Omega (u-f)^2~\dd x=\int_\Omega \left(\frac12-f\right)~u~\dd x+\frac12\int_\Omega f^2~\dd x.
\]
Thus, the binary valued problem
\begin{equation}\label{ROF_binary}
\min\limits_{u\in BV(\Omega;\{0,1\})} \frac12\int_\Omega (u-f)^2~\dd x +\beta~J(u) 
\end{equation} 
is equivalent to \eqref{problem1} with $g\equiv\frac12-f$. Consequently, Theorem \ref{mainthm1} applies, proving that a
minimizer of \eqref{problem1relaxed} with the above choice of $g$ and for $f\in L^2(\Omega)$ almost surely yields a solution of \eqref{ROF_binary} upon thresholding.


Note that in contrast to the analysis in \cite{analysisL1TV1} we do not need to confine the data $f$ to be binary, but rather consider arbitrary $L^2$-functions as data. In order to illustrate existence and uniqueness in the model, below we consider two special cases in $\Omega\in(-1,1)$.

\begin{example} Let $\beta<\frac14$ and $f$ be any function satisfying $f(x)=1$ for $x\leq0$, $f(x)<\frac14$ for $x>0$. Then, \eqref{ROF_binary} and 
\begin{equation}\label{ROF_binary2}
\min_{u\in BV(\Omega;[0,1])}\int_\Omega \left(\frac12-f\right)~u~\dd x+\beta~J(u)
\end{equation}
have a unique solution given (in an almost everywhere sense) by
\[
\hat{u}(x)=\left\{\begin{array}{ll}
         1,& \mathrm{if}\  x\leq0, \\
         0,& \mathrm{else.}
       \end{array}\right.
\]
This can be seen as follows. We have
\[
\int_\Omega \left(\frac12-f\right)~\hat{u}~\dd x+\beta~J(\hat{u})=\int_{-1}^0 \left(\frac12-f\right)~\dd x+\beta=\beta-\frac12.
\]
For arbitrary $u\in BV(\Omega;\{0,1\})$, we have
\begin{align*}
\int_\Omega \left(\frac12-f\right)~u~\dd x&=\int_{-1}^0\underbrace{\left(\frac12-f\right)}_{\leq 0}~\underbrace{u}_{\leq1}~\dd x+\int_{0}^1\underbrace{\left(\frac12-f\right)}_{>0}~\underbrace{u}_{\geq0}~\dd x\\ &\geq\int_{-1}^0\left(\frac12-f\right)~\dd x=-\frac12.
\end{align*}
Thus, a smaller objective value than for $\hat{u}$ can only be achieved for functions $u\in BV(\Omega;\{0,1\})$ with $J(u)<J(\hat{u})=1$. The latter is possible only for $u\equiv0$ or $u\equiv1$. For $u\equiv0$, we have 
\[
\int_\Omega \left(\frac12-f\right)~{u}~\dd x+\beta~J({u})=0>\beta-\frac12,
\]
and for $u\equiv1$ we find
\[
\int_\Omega \left(\frac12-f\right)~{u}~\dd x+\beta~J({u})=\int_{-1}^1\left(\frac12-f\right)~\dd x>-\frac14>\beta-\frac12.
\] 
Hence, $\hat{u}$ is the unique solution of 
\[
\min_{u\in BV(\Omega;\{0,1\})}\int_\Omega \left(\frac12-f\right)~u~\dd x+\beta~J(u).
\]
Due to our equivalence and uniqueness results $\hat{u}$ is also the unique minimizer of the relaxed problem \eqref{ROF_binary2} and of \eqref{ROF_binary}. 
\end{example}

\begin{example}
Let $0<\varepsilon<1$, $\beta<\frac{1-\varepsilon}{4}$, and $f$ be defined by
\[
f(x)=\left\{\begin{array}{ll}
         1,& \mathrm{if}\  x\leq-\varepsilon, \\
	\frac12, & \mathrm{if}\ -\varepsilon<x\leq\varepsilon,\\
         0,& \mathrm{if}\ x>\varepsilon.
       \end{array}\right.
\]
Then, each function $\hat{u}$ satisfying
\[
\hat{u}(x)=\left\{\begin{array}{ll}
         1,& \mathrm{if}\  x\leq-\varepsilon, \\
         0,& \mathrm{if}\ x>\varepsilon,
       \end{array}\right.
\]
is a solution of \eqref{ROF_binary2}.

In fact, for each $\hat{u}$ defined as above, we have
\[
\int_\Omega \left(\frac12-f\right)~\hat{u}~\dd x+\beta~J(\hat{u})=\int_{-1}^{-\varepsilon} \left(\frac12-f\right)~\dd x+\beta=\beta-\frac12(1-\varepsilon).
\]
Due to the above results, the relaxed problem \eqref{ROF_binary2} has always a solution in $BV(\Omega;\{0,1\})$, so that it suffices to consider functions in this class. For arbitrary $u\in BV(\Omega;\{0,1\})$, we have
\begin{align*}
\int_\Omega \left(\frac12-f\right)~u~\dd x&=\int_{-1}^{-\varepsilon}\underbrace{\left(\frac12-f\right)}_{\leq 0}~\underbrace{u}_{\leq1}~\dd x+\int_{-\varepsilon}^1\underbrace{\left(\frac12-f\right)}_{\geq0}~\underbrace{u}_{\geq0}~\dd x\\ &\geq\int_{-1}^{-\varepsilon}\left(\frac12-f\right)~\dd x=-\frac12(1-\varepsilon).
\end{align*}
Again, a smaller objective value than for $\hat{u}$ can only be achieved for functions $u\in BV(\Omega;\{0,1\})$ with $J(u)<J(\hat{u})=1$, i.e., $u\equiv0$ or $u\equiv1$. For $u\equiv0$, we have 
\[
\int_\Omega \left(\frac12-f\right)~{u}~\dd x+\beta~J({u})=0>\beta-\frac{1-\varepsilon}{2},
\]
and for $u\equiv1$ we obtain
\[
\int_\Omega \left(\frac12-f\right)~{u}~\dd x+\beta~J(u)=0>\beta-\frac{1-\varepsilon}{2}.
\] 
Thus, $\hat{u}$ is a solution of \eqref{ROF_binary2}. Furthermore, we readily conclude that each function
\[
u^\nu(x)=\left\{\begin{array}{ll}
         1,& \mathrm{if}\  x\leq\nu, \\
         0,& \mathrm{if}\ x>\nu,
       \end{array}\right.
\]
is a solution of \eqref{ROF_binary} if $\nu\in(-\varepsilon,\varepsilon)$. Therefore, in this case there is no unique reconstruction of $f$. 
\end{example}

Now we study the behavior of our algorithm for restoring binary images corrupted by Gaussian noise. In Figure \ref{result_ROF} we show the results obtained by solving \eqref{dual2} for restoring degraded images with 30\% and 50\% Gaussian noise, respectively. For all tests related to this example we fix $c=100$, $\epsilon=10^{-7}$, $\gamma=0.1$, and $\alpha=10^3$. We stop the semismooth Newton iteration as soon as the initial residual is reduced by a factor of $10^{-8}$, or the difference between two successive residuals was smaller than $10^{-8}$. For solving the linear system occurring in each Newton step we utilize the PCG-method. The preconditioner consists of the Laplacian and the terms involving the indicator functions of $\mathcal{A}_5^k$ and $\mathcal{A}_6^k$. The stopping tolerance for the PCG-method in iteration $k+1$ is given by
\[ 
\mathrm{tol}_{k+1}=10^{-3}\cdot\min\left\{\left(\frac{\mathrm{res}_k}{\mathrm{res}_0}\right)^{3/2},\frac{\mathrm{res}_k}{\mathrm{res}_0}\right\},
\]
where $\mathrm{res}_k$ denotes the residual of the linear system \eqref{ELeqn} at $\vec{q}_k$. Thus, we use an inexact semismooth Newton iteration.

For the reconstruction we explore different values of $\beta$, which controls the trade-off between a good fit of $f$ and a smoothness requirement due to the total variation regularization. From the results shown in Figure \ref{result_ROF}, we find that small $\beta$ preserves details, but at the same time some noticeable noise remains; otherwise, if $\beta$ becomes large, details are over-regularized. As we fix sufficiently small $\epsilon$ and large enough $c$ in our tests, we obtain the expected binary reconstruction $u$ directly. This confirms our earlier observation in connection with \eqref{eqn_u}. In all other cases, we simply set $t=0.5$ and obtain the binary result $u^t$ as proposed in Theorem \ref{mainthm1}. Concerning the algorithmic behavior, from the residual plots we observe that despite the inexact linear system solves our algorithm only requires a rather small numbers of iterations until successful termination. This fast convergence can be attributed to the use of a semismooth Newton solver.
Moreover, the iteration numbers are found to be stable with respect to different noise levels. 

\begin{figure}[t]
\begin{center}
\begin{minipage}[t]{3cm}
\includegraphics[height=3cm]{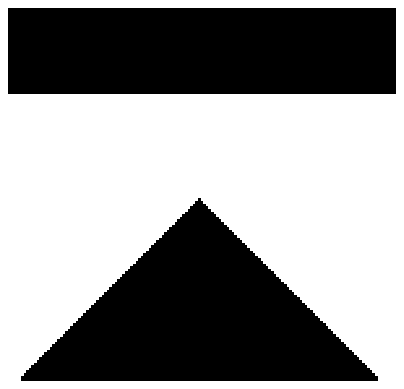}\\
\centering{(a)}
\end{minipage}
\begin{minipage}[t]{3cm}
\includegraphics[height=3cm]{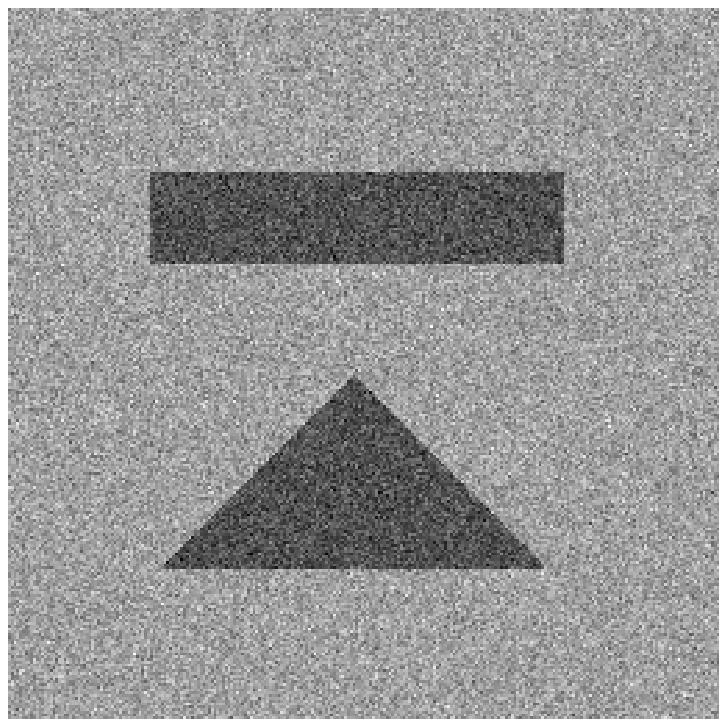}\\
\centering{(b)}
\end{minipage}
\begin{minipage}[t]{3cm}
\includegraphics[height=3cm]{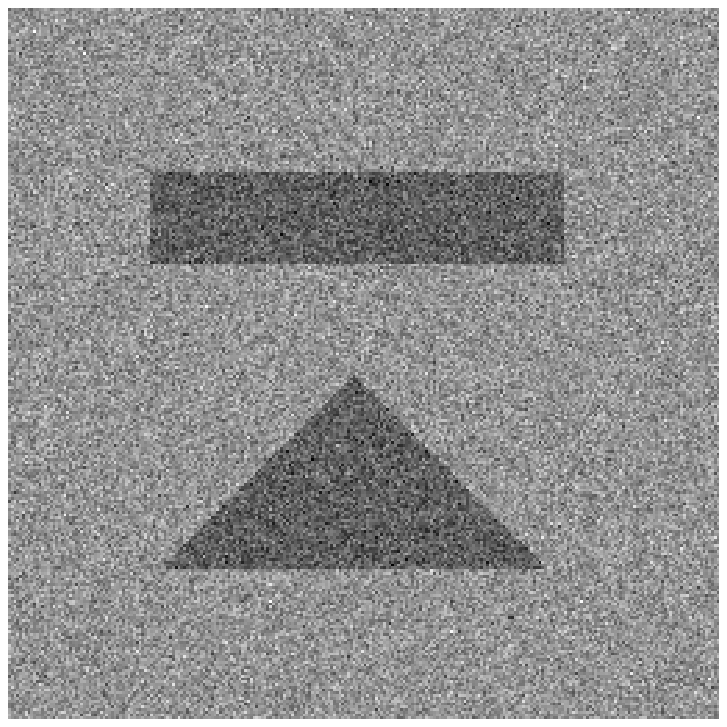}\\
\centering{(c)}
\end{minipage}\\
\begin{minipage}[t]{3cm}
\includegraphics[height=3cm]{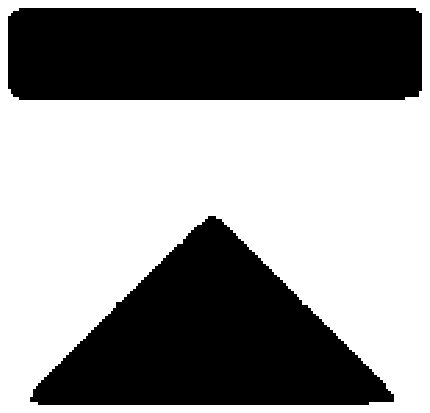}
\end{minipage}
\begin{minipage}[t]{3cm}
\includegraphics[height=2.5cm]{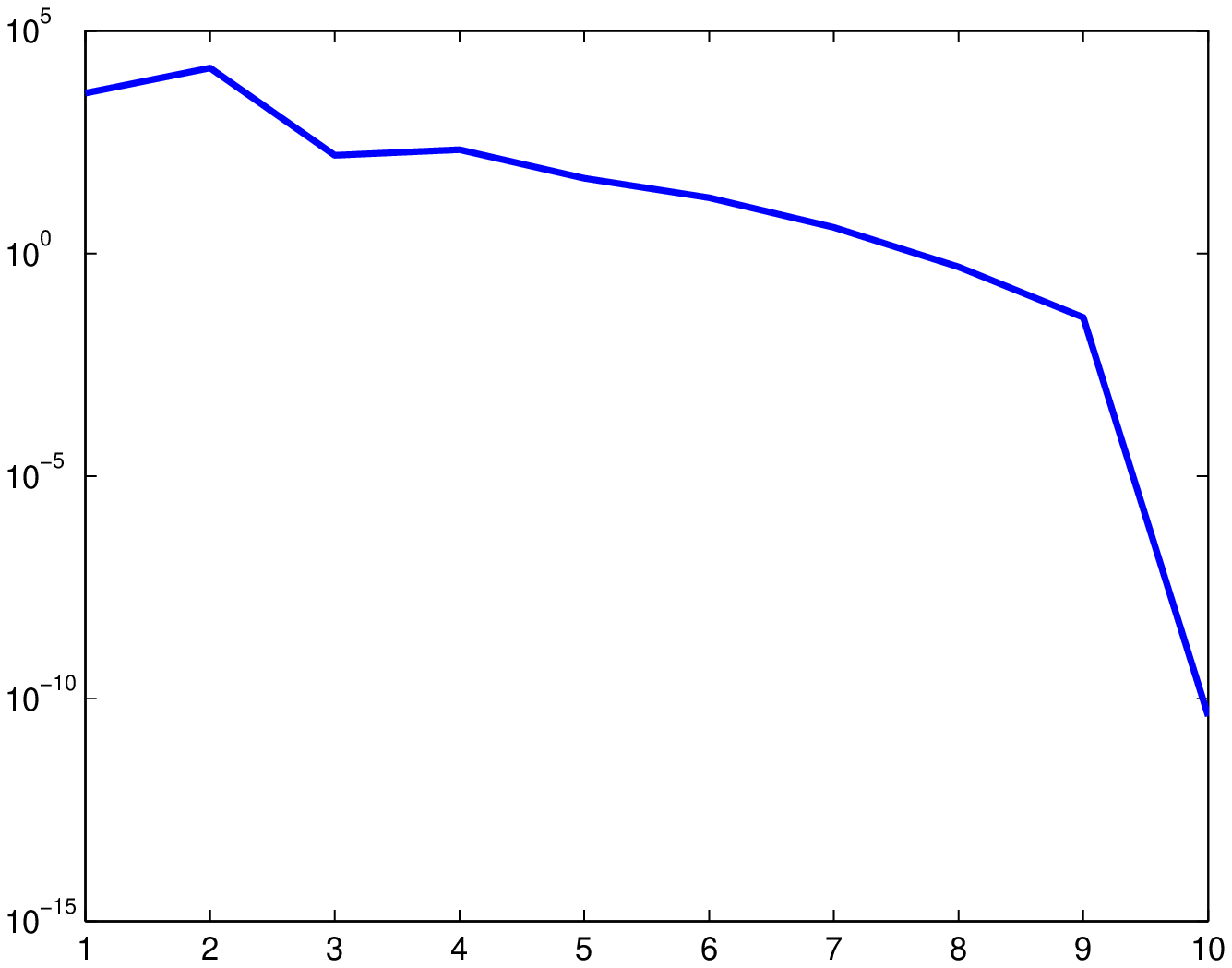}
\end{minipage}
\begin{minipage}[t]{3cm}
\includegraphics[height=3cm]{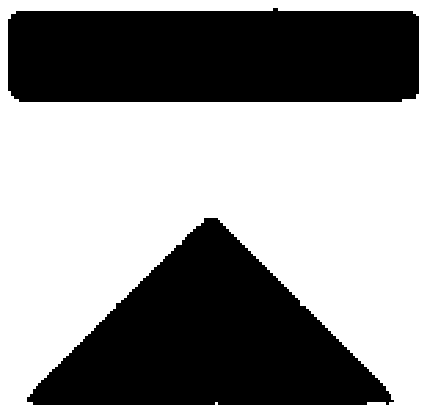}
\end{minipage}
\begin{minipage}[t]{3cm}
\includegraphics[height=2.5cm]{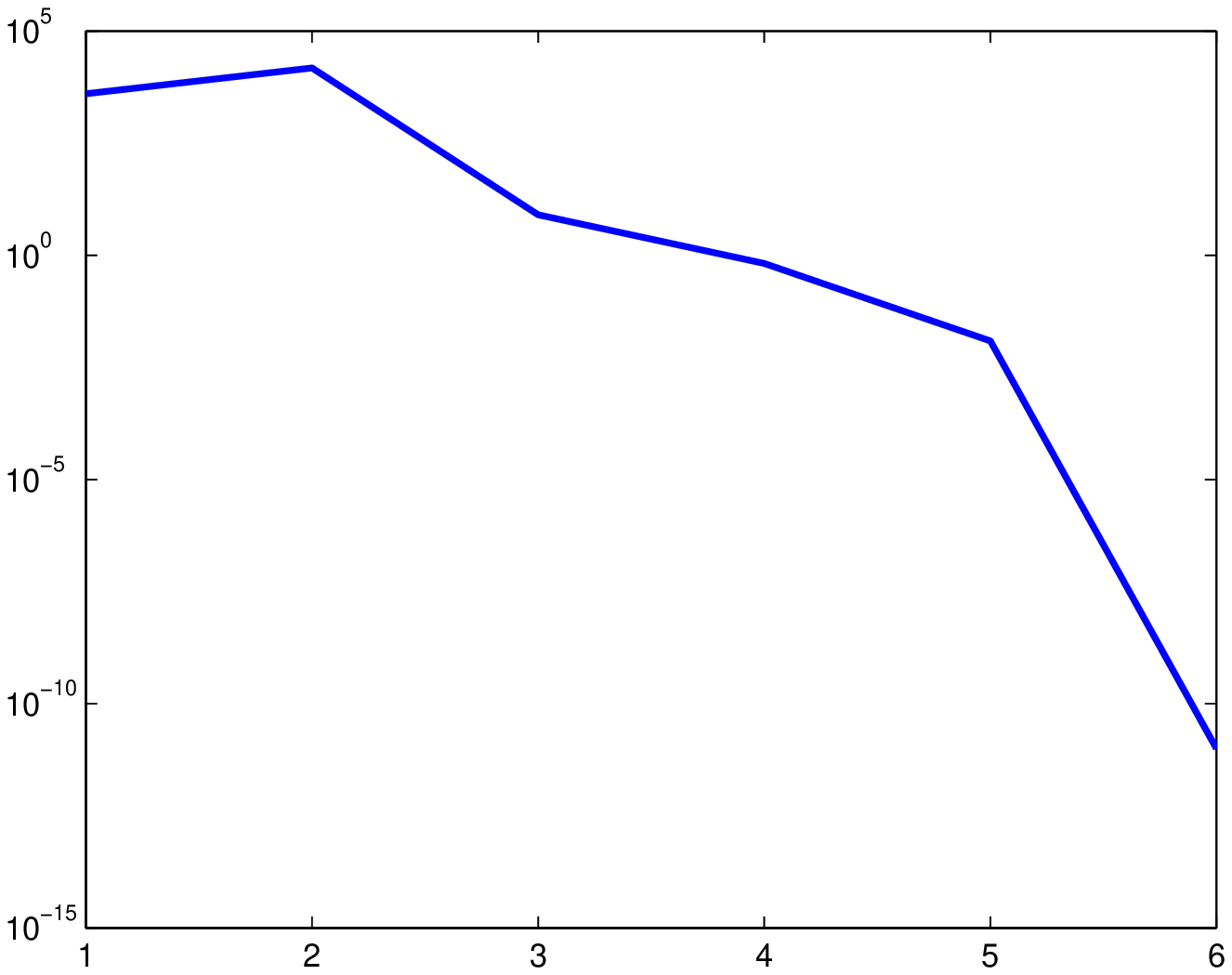}
\end{minipage}\\
\begin{minipage}[t]{3cm}
\includegraphics[height=3cm]{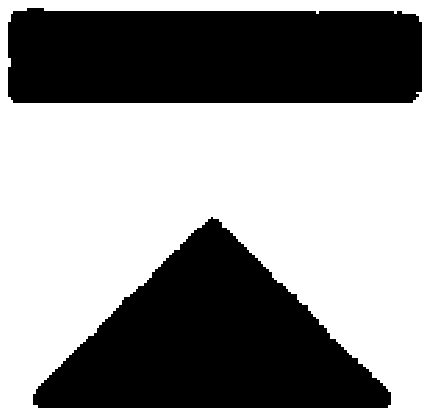}\\
\centering{(d)}
\end{minipage}
\begin{minipage}[t]{3cm}
\includegraphics[height=2.5cm]{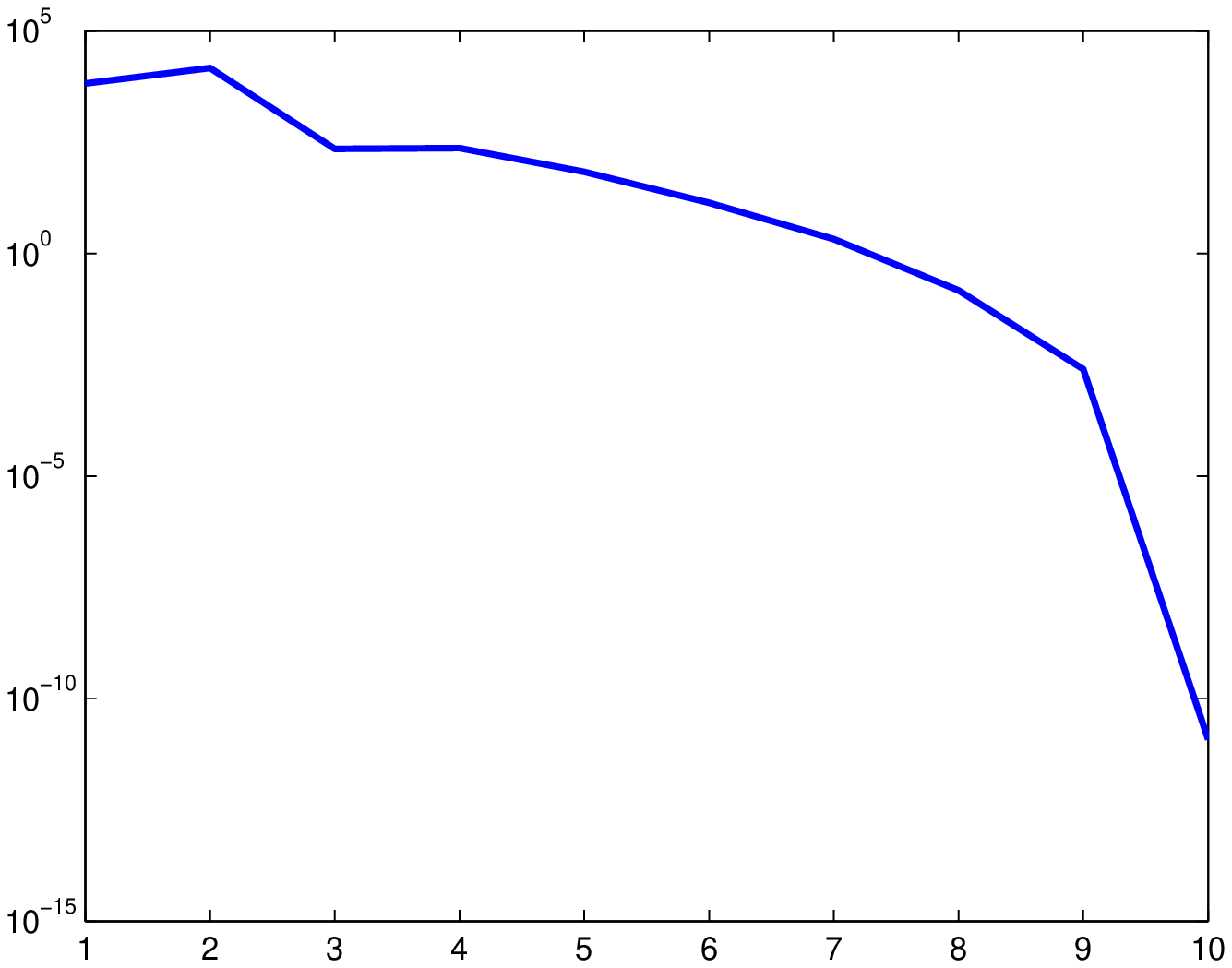}\\
\centering{(e)}
\end{minipage}
\begin{minipage}[t]{3cm}
\includegraphics[height=3cm]{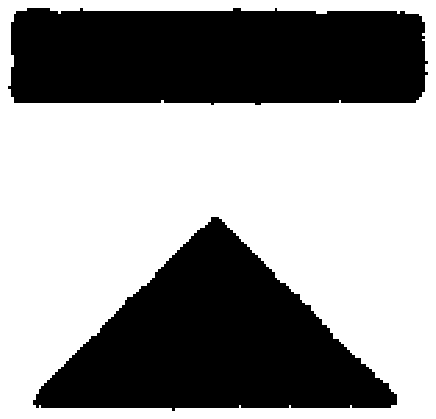}\\
\centering{(f)}
\end{minipage}
\begin{minipage}[t]{3cm}
\includegraphics[height=2.5cm]{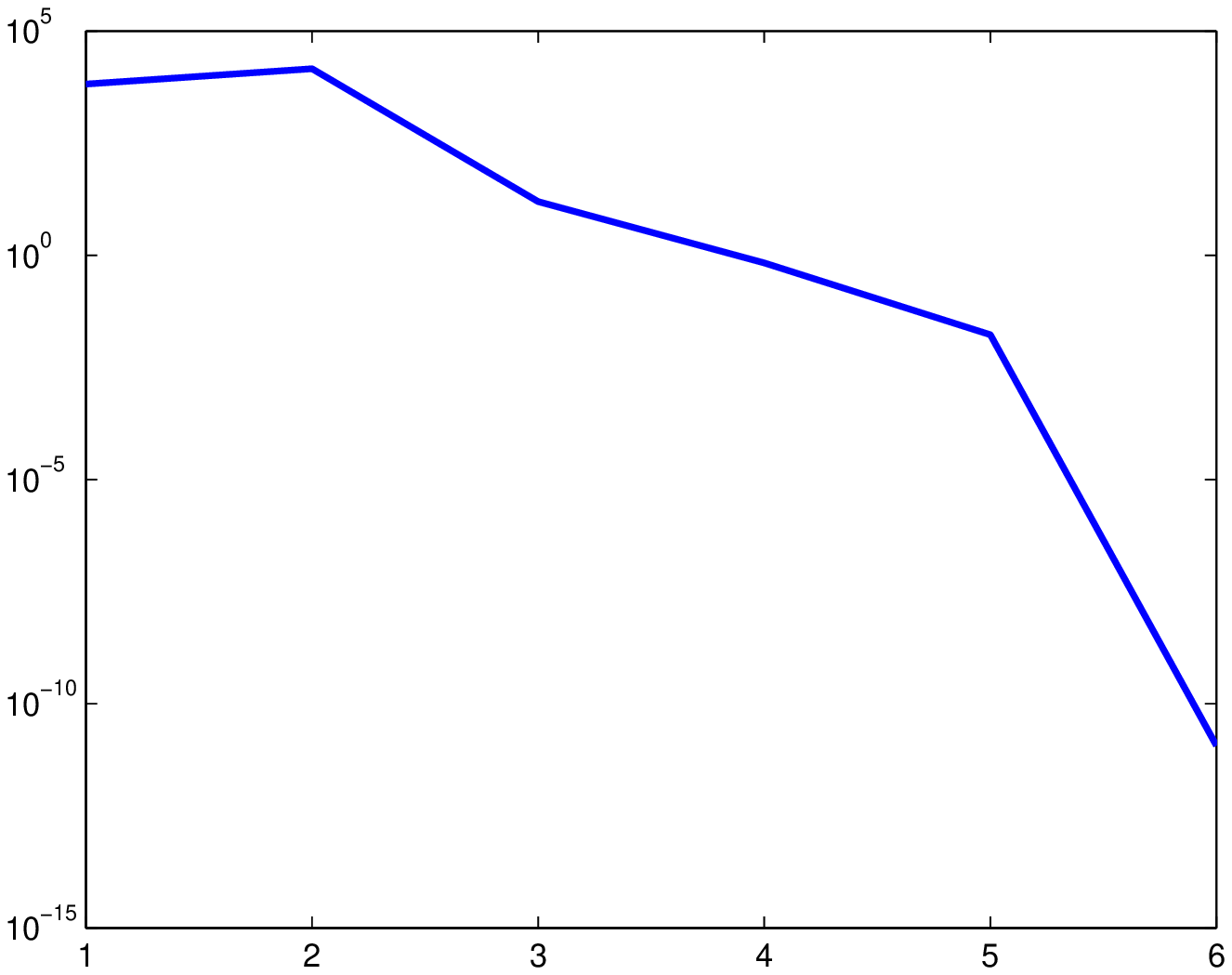}\\
\centering{(g)}
\end{minipage}
\end{center}
\caption{\small\label{result_ROF}(a) Original image, (b) noisy image with 30\% Gaussian noise, (c) noisy image with 50\% Gaussian noise. The results shown in the second row are from 30\% noise; the ones shown in the third row are from 50\% noise. (d) Restored images with $\beta=10^{-3}$, (e) the corresponding residual plots, (f) restored images with $\beta=10^{-5}$, (g) the corresponding residual plots.}
\end{figure}

\subsection{Mumford-Shah-based binary image segmentation}

In \cite{Chan_Vese}, an active contour model based on Mumford-Shah segmentation techniques was proposed. It distinguishes objects from the background by solving
\begin{equation}\label{MS_model}
\min\limits_{\scriptsize\begin{array}{c} u\in BV(\Omega;\{0,1\}),\\ c_1\in\mathbb{R},c_2\in\mathbb{R}\end{array}} \int_\Omega \big[u(c_1-f)^2+(1-u)(c_2-f)^2\big]~\dd x +\beta~J(u),
\end{equation}
where $u\in BV(\Omega;\{0,1\})$ is the indicator function of the set $\Omega^{obj}$, which corresponds to the detected objects, and $\Omega\setminus\Omega^{obj}$ refers to the background set. Thus, for such a $u$ the model \eqref{MS_model} yields a piecewise constant approximation $f^{MS}$ of $f$ with $f^{MS}=c_1$ on $\Omega^{obj}$ and $f^{MS}=c_2$ on $\Omega\setminus\Omega^{obj}$. Here, $c_1$ and $c_2$ are the average values of $f$ on the set $\Omega^{obj}$ and $\Omega\setminus\Omega^{obj}$, respectively. 

Obviously the key step to solving \eqref{MS_model} is to obtain the optimal indicator function $u$. We note that for fixed $c_1$ and $c_2$ the minimization problem \eqref{MS_model} with respect to $u$ is equivalent to \eqref{problem1} with $g\equiv(c_1-f)^2-(c_2-f)^2$. Hence, we can apply our framework to solve this problem. 

In Figure \ref{result_MS}, we display the resulting solutions $u$ when segmenting the specified test images with different choices for $\beta$. In all tests, the choices of all other parameters and the stopping rules are as in the previous section. Initially, we set $u^0(x)=1$ if $f(x)>0.5$ and $u^0(x)=0$ otherwise. This gives $\Omega^{obj}=\{x\in\Omega: f(x)>0.5\}$. Based on the initial $u^0$, we then calculate $c^0_1$ and $c^0_2$ by averaging the data as discussed above.  Algorithmically, whenever we obtain a new $u^k$ by solving \eqref{problem1}, then we update $c^k_1$ and $c^k_2$ based on $u^k$ and go to the next updating step for $u$ until $\|u^k-u^{k-1}\|_1\leq10^{-4}|\Omega|$. Since the total variation of an indicator function on a set is related to the perimeter of the pertinent set \cite{BVBook2}, the resulting segmentation $u$ strongly depends on the selection of $\beta$. If $\beta$ is small, then small-scale objects will also be detected. On the other hand, if $\beta$ is large, then only large-scale objects or objects formed by grouping smaller scale features are usually detected. This behavior can be seen from the results depicted in Figure \ref{result_MS}.  

\begin{figure}[t]
\begin{center}
\begin{minipage}[t]{3.2cm}
\includegraphics[height=3cm]{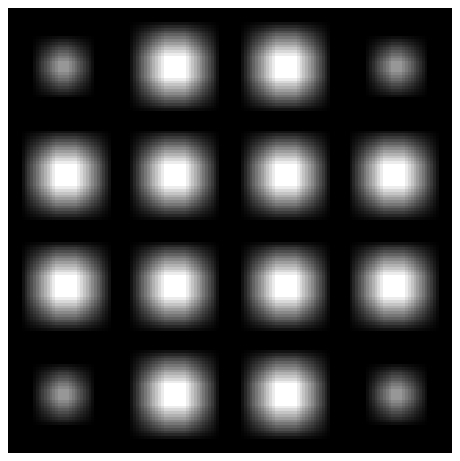}
\end{minipage}
\begin{minipage}[t]{3.2cm}
\includegraphics[height=3cm]{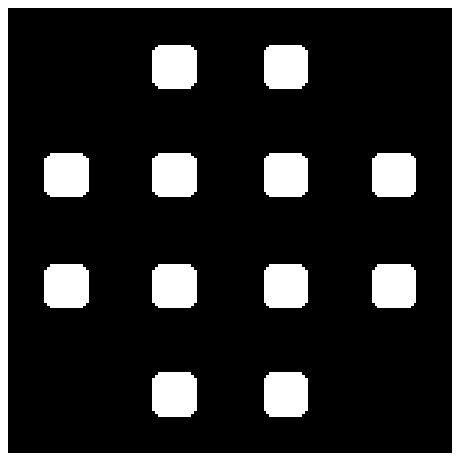}
\end{minipage}
\begin{minipage}[t]{3.2cm}
\includegraphics[height=3cm]{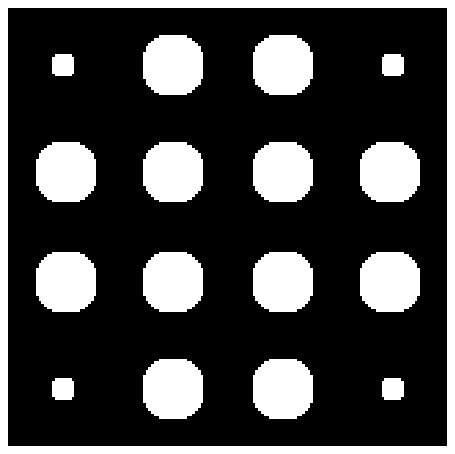}
\end{minipage}\\
\hspace{-2mm}\begin{minipage}[t]{3.2cm}
\includegraphics[height=3.3cm]{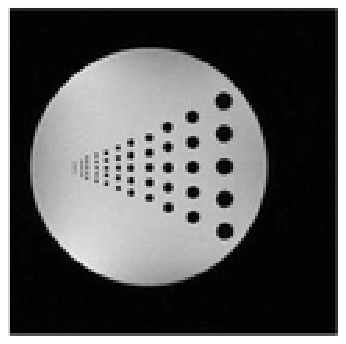}\\
\centering{(a)}
\end{minipage}
\begin{minipage}[t]{3.2cm}
\includegraphics[height=3.3cm]{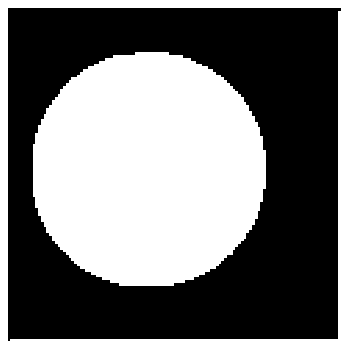}\\
\centering{(b)}
\end{minipage}
\begin{minipage}[t]{3.2cm}
\includegraphics[height=3.3cm]{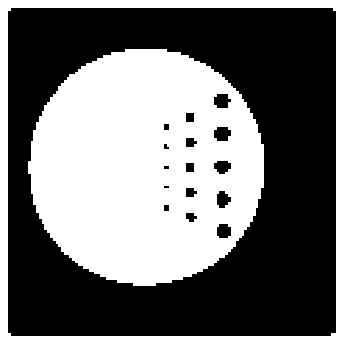}\\
\centering{(c)}
\end{minipage}
\end{center}
\caption{\small\label{result_MS}The results $u$ for Mumford-Shah-based image segmentation with different choices of $\beta$. (a) The test images $f$ for segmentation, (b) the results $u$ with $\beta=8\times10^{-3}$, (c) the results $u$ with $\beta=8\times10^{-5}$.}
\end{figure}

Although the original problem \eqref{problem1} is nonconvex, its relaxation \eqref{problem1relaxed} is convex. The penalized optimization problem \eqref{unmin2}, which is solved by our algorithm, is even strictly convex. Hence, it has a unique global minimizer. However, the model \eqref{MS_model} is non-convex with respect to $(u, c_1, c_2)$, even after relaxation. But numerically in Figure \ref{result_MS2} we demonstrate that our algorithm works reliably and converges to the same minimizer, regardless of the initial choice $u^0$. We also mention that generalizations of the model \eqref{MS_model} as proposed in \cite{MS_Burger} may be treated with our approach, too.

\begin{figure}[t]
\begin{center}
\begin{minipage}[t]{3.2cm}
\includegraphics[height=3cm]{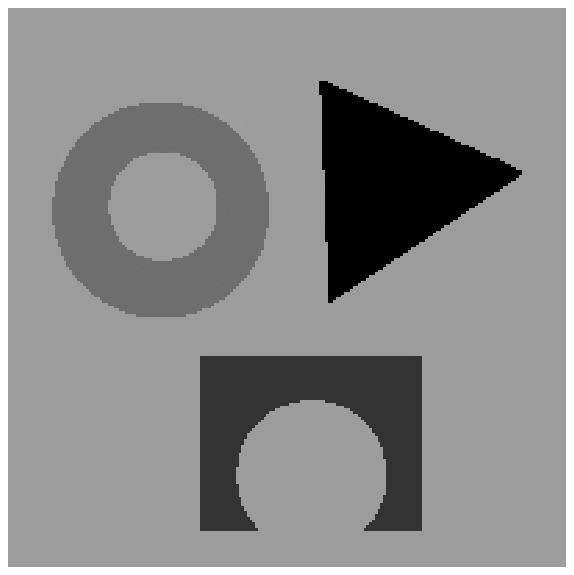}\\
\centering{(a)}
\end{minipage}\\
\begin{minipage}[t]{3.2cm}
\includegraphics[height=3cm]{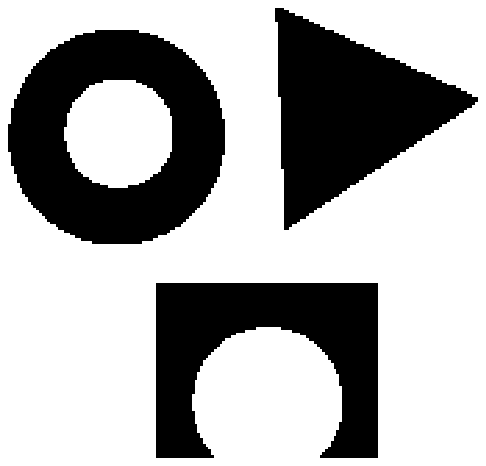}
\end{minipage}
\begin{minipage}[t]{3.2cm}
\includegraphics[height=3cm]{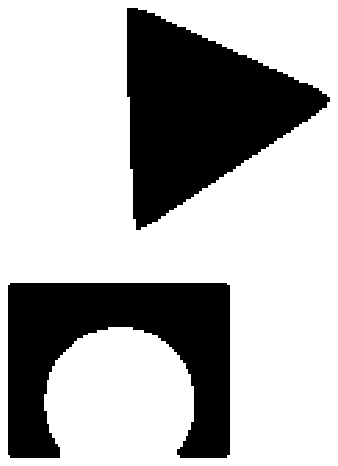}
\end{minipage}
\begin{minipage}[t]{3.2cm}
\includegraphics[height=3cm]{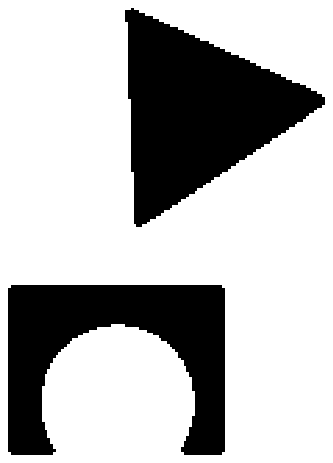}
\end{minipage}\\
\begin{minipage}[t]{3.2cm}
\includegraphics[height=3cm]{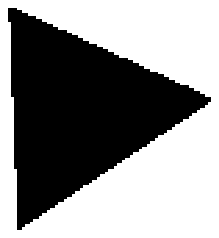}\\
\centering{(b)}
\end{minipage}
\begin{minipage}[t]{3.2cm}
\includegraphics[height=3cm]{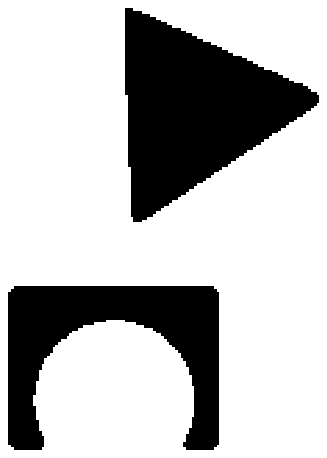}\\
\centering{(c)}
\end{minipage}
\begin{minipage}[t]{3.2cm}
\includegraphics[height=3cm]{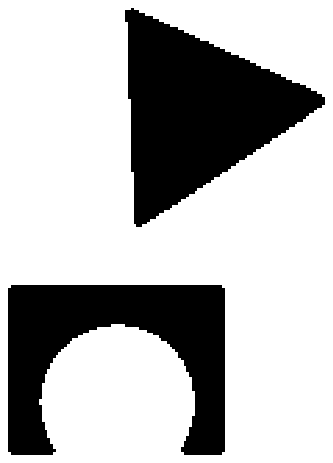}\\
\centering{(d)}
\end{minipage}
\end{center}
\caption{\small\label{result_MS2}The results $u$ for Mumford-Shah-based image segmentation with different initial values $u^0$ and $\beta=10^{-5}$. (a) The test image $f$ for segmentation, (b) the initial $u^0$, (c) $u^1$, (d) the results $u$.}
\end{figure}

\subsubsection{Multi-phase piecewise constant Mumford-Shah model for multi-labeling}

Instead of only separating objects from a background, the multi-phase piecewise constant Mumford-Shah image segmentation problem distinguishes between more than two sets $\Omega_i$ with associated values $c_i$, i.e., it formulates multi-labeling. In this context, we want to represent the input image $f$ by a piecewise constant function $f^{pc}:\Omega\rightarrow\{c_1,\cdots,c_N\}$ with the partitioning $\Omega=\bigcup_{i=1}^N\Omega_i$ and the property that $f^{pc}(x)=c_i$ if $x\in\Omega_i$ for $i=1,\cdots,N$. Introducing an indicator function $u_i\in BV(\Omega;\{0,1\})$ of the set $\Omega_i$ for each $i\in\{1,\cdots,N\}$, we have $f^{pc}=\sum_{i=1}^N c_iu_i$. Extending the two-phase piecewise constant Mumford-Shah model \eqref{MS_model} to the multi-phase case, we utilize the multi-labeling model
\begin{equation}\label{ML_model}
\min\limits_{\scriptsize\begin{array}{c} \vec{u}\in BV(\Omega;\{0,1\})^N,\vec{c}\in\mathbb{R}^N\\ \sum_{i=1}^N u_i(x)=1\ \mbox{for all}\ x\in\Omega\end{array}} \sum_{i=1}^N\int_\Omega  u_i~(c_i-f)^2~\dd x +\beta~\sum_{i=1}^NJ(u_i),
\end{equation}
where $\vec{u}=(u_1,\cdots,u_N)^\top$ and $\vec{c}=(c_1,\cdots,c_N)^\top$. In fact, \eqref{ML_model} is a continuous variant of the Potts model \cite{Potts} with total variation regularization as perimeter penalization. 

Considering the convex relaxation of \eqref{ML_model} in a rather straightforward manner gives
\begin{equation}\label{ML_model2}
\min\limits_{\scriptsize\begin{array}{c} \vec{u}\in BV(\Omega;[0,1])^N,\vec{c}\in\mathbb{R}^N\\ \sum_{i=1}^N u_i(x)=1\ \mbox{for all}\ x\in\Omega\end{array}} \sum_{i=1}^N\int_\Omega  u_i~(c_i-f)^2~\dd x +\beta~\sum_{i=1}^NJ(u_i).
\end{equation}
We emphasize that this relaxation is known to be not necessarily exact; see \cite{Zach_ML} for more details in this vein. As a consequence, it is indeed possible that the minimum of \eqref{ML_model2} is strictly less than the minimum of \eqref{ML_model}, and, thus, in general one cannot expect to obtain a minimizer of \eqref{ML_model} from a minimizer of \eqref{ML_model2} based on Theorem \ref{mainthm1}. Further we mention that recently various research efforts have been devoted to studying and solving the multi-labeling problem; see \cite{ML_SIIMS10,ML_CVPR,ML_Chapter,ML_IJCV,ML_SIIMS11}.

In order to provide an exact convex relaxation context for the multi-phase case, we assume that the number of phases is equal to or less than $2^M$. Referring to the work in \cite{ML_CV}, such a situation can be represented by $M$ indicator functions to represent the partitioning. As a consequence, the model \eqref{MS_model} can be generalized to
\begin{equation}\label{ML_model3}
\min\limits_{ \vec{u}\in BV(\Omega;\{0,1\})^M,\vec{c}\in\mathbb{R}^{2^M}} \sum_{\vec{b}\in\{0,1\}^M}\int_\Omega  Z_{\vec{b}}(\vec{u})(c_{\vec{b}}-f)^2~\dd x +\beta~\sum_{i=1}^MJ(u_i),
\end{equation}
where $\vec{u}=(u_1,\cdots,u_M)^\top$ with $u_i\in BV(\Omega;\{0,1\})$ for each $i=1,\cdots,M$, and
\[
Z_{\vec{b}}(\vec{u})=\prod_{i=1}^M z_{b_i}(u_i), \qquad z_{b_i}(y)=\left\{\begin{array}{ll}1-y, & \mbox{if}\ b_i=0,\\ y, & \mbox{if}\ b_i=1.\end{array}\right. 
\]
It is straightforward to prove that \eqref{ML_model3} has an exact convex relaxation (in the sense of Theorem \ref{mainthm1})
\begin{equation}\label{ML_model4}
\min\limits_{\vec{u}\in BV(\Omega;[0,1])^M,\vec{c}\in\mathbb{R}^{2^M}} \sum_{\vec{b}\in\{0,1\}^M}\int_\Omega  Z_{\vec{b}}(\vec{u})(c_{\vec{b}}-f)^2~\dd x +\beta~\sum_{i=1}^M J(u_i),
\end{equation}
and a solution of \eqref{ML_model3} can be obtained from a solution of \eqref{ML_model4} in the same way as in Theorem \ref{mainthm1}. We solve \eqref{ML_model4} iteratively as follows: For $k\geq 0$ and for $i=1,2,\ldots,M$, each $u_i^k$ is computed by fixing $u_j=u_j^{k}$ for all $1\leq j<i$ and $u_j=u_j^{k-1}$ for $i<j\leq M$. With such a setting, the minimization problem  resulting from \eqref{ML_model4} fits into our framework \eqref{problem1}. Hence, we can employ our proposed solution algorithm. In Figure \ref{ML_result1} we show an example for labelling an image by 2 and 3 indicator functions, respectively.

\begin{figure}[t]
\begin{center}
\begin{minipage}[t]{4.5cm}
\includegraphics[height=3cm]{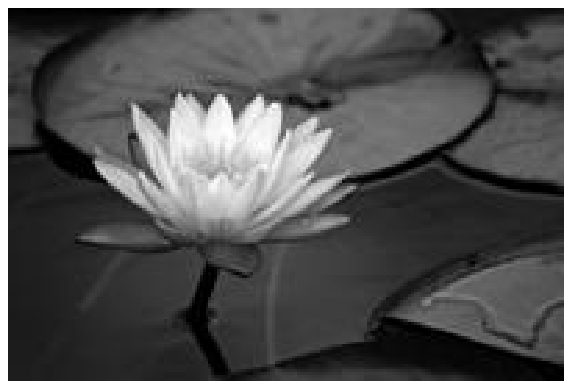}\\
\centering{(a)}
\end{minipage}
\begin{minipage}[t]{4.5cm}
\includegraphics[height=3cm]{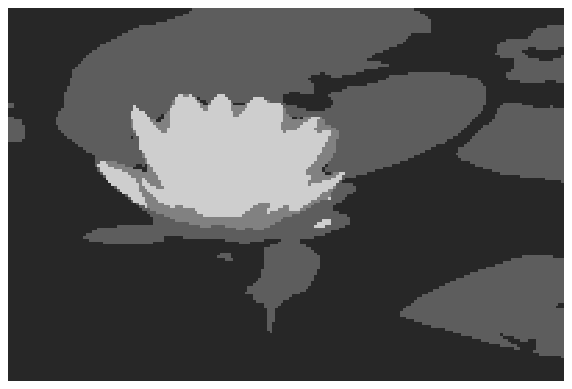}\\
\centering{(b)}
\end{minipage}
\begin{minipage}[t]{4.5cm}
\includegraphics[height=3cm]{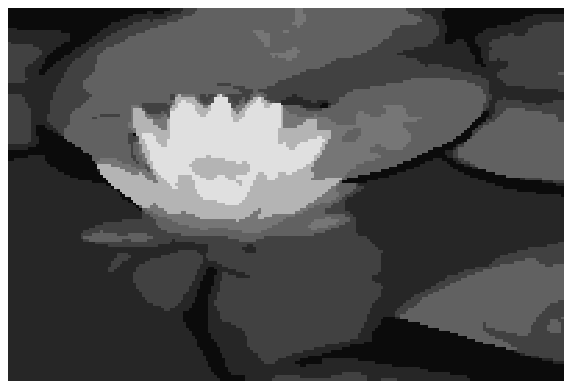}\\
\centering{(c)}
\end{minipage}
\end{center}
\caption{\small\label{ML_result1} (a) The test image $f$ for multilabeling, (b) $f^{pc}$ from solving the 4-phase model, i.e., \eqref{ML_model3} with $M=2$, (c) $f^{pc}$ from solving the 8-phase model, i.e., \eqref{ML_model3} with $M=3$.}
\end{figure}

Although the objective function in \eqref{ML_model3} is convex for each $u_i$ with all the other $u_j$ ($j\neq i$) fixed, the associated overall minimization problem is of course not convex. Thus, the algorithm is not guaranteed to converge to a global minimizer. Moreover, the obtained results may depend on the initial choice of $\vec{u}$. In order to show this influence of the initialization we depict several different initial choices together with the obtained solutions in Figure \ref{ML_result2}.

\begin{figure}[t]
\begin{center}
\begin{minipage}[t]{3.2cm}
\includegraphics[height=3cm]{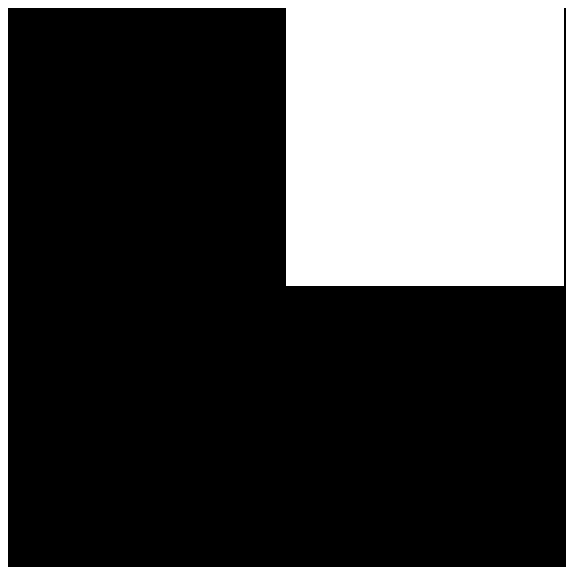}
\end{minipage}
\begin{minipage}[t]{3.2cm}
\includegraphics[height=3cm]{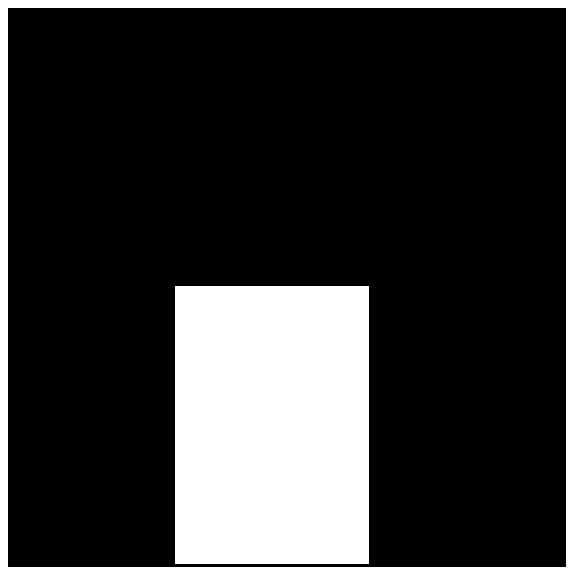}
\end{minipage}
\begin{minipage}[t]{3.2cm}
\includegraphics[height=3cm]{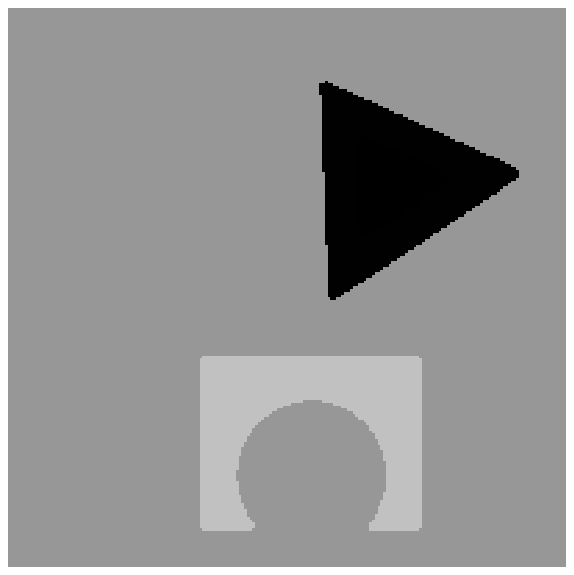}
\end{minipage}\\
\begin{minipage}[t]{3.2cm}
\includegraphics[height=3cm]{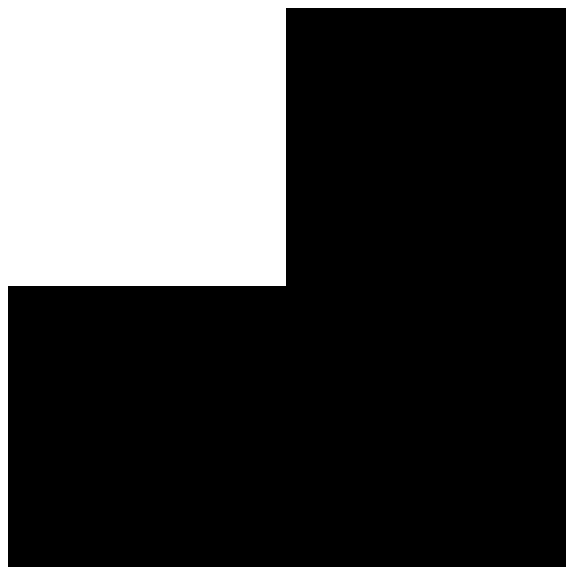}\\
\centering{(a)}
\end{minipage}
\begin{minipage}[t]{3.2cm}
\includegraphics[height=3cm]{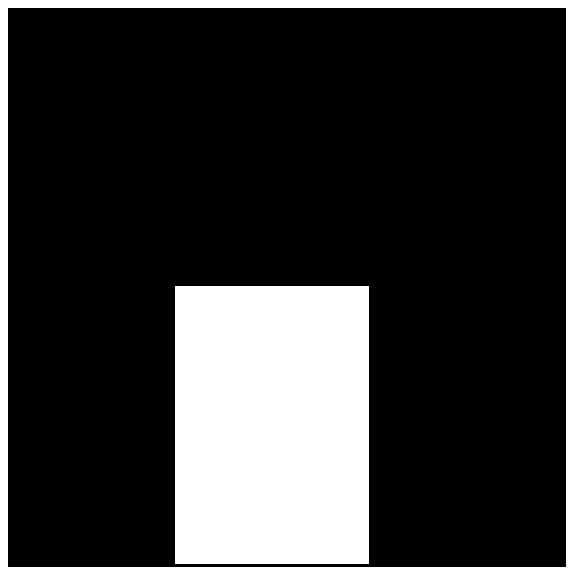}\\
\centering{(b)}
\end{minipage}
\begin{minipage}[t]{3.2cm}
\includegraphics[height=3cm]{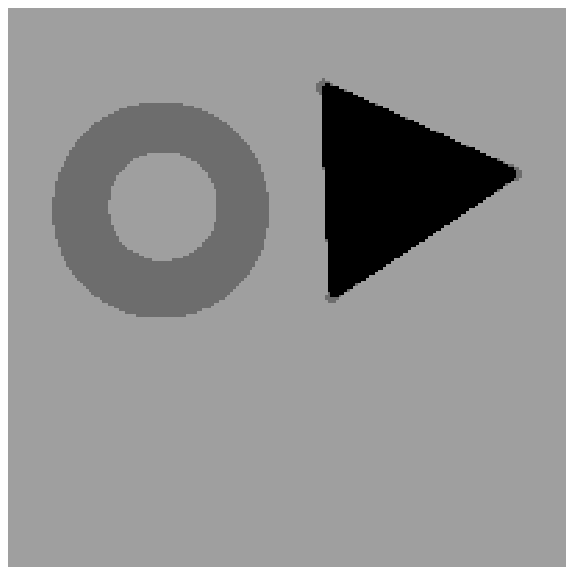}\\
\centering{(c)}
\end{minipage}
\end{center}
\caption{\small\label{ML_result2}The results $f^{pc}$ from solving \eqref{ML_model3} with different initial values $\vec{u}$. (a) the initial $u_1$, (b) the initial $u_2$, (c) the results $f^{pc}$.}
\end{figure}

\subsection{Further applications}
In this section, we provide two more applications fitting our exact relaxation framework, and show their relation to the full model \eqref{full_model}.

\subsubsection{Optimal design of composite membranes}

A well-known problem in the design of composite membranes (cf. \cite{OD_1,OD_2}) is the minimization of the first eigenvalue $\lambda^o$ of the elliptic problem
\[
-\bigtriangleup v+\alpha~u~v=\lambda^o~v
\]
in a domain $\Omega$ subject to homogeneous Dirichlet boundary conditions on $\partial\Omega$ and a volume constraint $\int_\Omega u~\dd x=V$, where $v$ is the eigenfunction. The minimization is carried out with respect to $u$ being the indicator function of some subset of $\Omega$, which is termed an optimal configuration for the data $(\Omega, \alpha, V)$. If we use additional standard perimeter penalization and the Rayleigh variational principle for the first eigenvalue, then the associated optimization problem becomes
\begin{align}\label{OD_1}
 &\min\limits_{(u,v)\in BV(\Omega;\{0,1\})\times H^1_0(\Omega)}\quad \frac{\gamma^o}{2}\frac{\int_\Omega \big(|\nabla v|^2+\alpha ~u~v^2\big)~\dd x}{\int_\Omega v^2~\dd x}+\beta~J(u)\\
&\mbox{s.t.}  \quad \int_\Omega u\ dx\leq V,\nonumber
\end{align}
for fixed $\gamma^o>0$. Here, $H^1_0(\Omega)$ denotes the Sobolev space with distributional derivatives in $L^2(\Omega)^d$ and vanishing trace on $\partial\Omega$.

The problem fits into the framework of \eqref{full_model} with the choice $\mathcal{V}=H^1_0(\Omega)$, the functional
\[
F(v):=\frac{\gamma^o}{2}\frac{\int_\Omega \big(|\nabla v|^2\big)~\dd x}{\int_\Omega v^2~\dd x}
\]
and the operator
\[
G(v):=\frac{\gamma^o}{2}\frac{v^2}{\int_\Omega v^2~\dd x}.
\]
Note that due to the compact embedding of Sobolev spaces, $G$ is a strongly continuous operator from $\mathcal{V}$ to $L^2(\Omega)$. Hence, we conclude that a minimizer of the composite membrane problem can be computed by solving the relaxed problem
\begin{align*}
& \min\limits_{(u,v)\in BV(\Omega;[0,1])\times H^1_0(\Omega)}\quad \frac{\gamma^o}{2}\frac{\int_\Omega \big(|\nabla v|^2+\alpha ~u~v^2\big)~\dd x}{\int_\Omega v^2~\dd x}+\beta~J(u)\\
&\mbox{s.t.}  \quad \int_\Omega u\ dx\leq V,
\end{align*}
and subsequently thresholding to
\[
u^t(x)=\left\{\begin{array}{ll}
         1,& \mathrm{if}\  u>t, \\
         0,& \mathrm{else.}
       \end{array}\right.
\]

\subsubsection{Minimal compliance topology optimization}\label{ssec:mincompliance}

A particularly important problem in structural topology optimization is the minimal compliance problem (cf. \cite{TOP_Book2,TOP_Book}), which can be written as 
\begin{align}\label{TOP_1}
& \min\limits_{(u,v)\in BV(\Omega;\{0,1\})\times H^1_D(\Omega;\mathbb{R}^d)}\quad l(v)\\
&\mbox{s.t.} \quad a_u(v,\tilde{v})=l(\tilde{v}), \quad \mbox{for all $\tilde{v}\in H^1_D(\Omega;\mathbb{R}^d)$},\nonumber\\
  &\int_\Omega u~\dd x\leq V,\nonumber
\end{align}
where $u$ denotes the (material) density, $v$ is the elastic displacement, and $H^1_D(\Omega;\mathbb{R}^d)$ is the subspace of functions $v\in H^1(\Omega;\mathbb{R}^d)$ with vanishing trace on $\varGamma_D\subset \partial\Omega$ with $\partial\Omega$ the boundary of the bounded domain $\Omega$. In addition we have
\begin{align*}
 l(\tilde{v})&:=\int_\Omega f~\tilde{v}~\dd x+\int_{\varGamma_T} t~\tilde{v}~\dd s,\\
a_u(v,\tilde{v})&:=\int_\Omega u~(E(v):\mathbf{C}:E(\tilde{v}))~\dd x,
\end{align*}
where $f$ are body forces, $t$ is the boundary traction on the traction part $\varGamma_T\subset\partial\Omega$, $\mathbf{C}$ denotes a given elasticity tensor, and $E(v):=\frac12(\nabla v+\nabla v^\top)$. The main idea of the formulation is to maximize the material stiffness at a given volume.

Our algorithm can be applied in the minimal compliance context upon approximating \eqref{TOP_1} by
\begin{align}\label{TOP_2}
 \min\limits_{(u,v)\in BV(\Omega;\{0,1\})\times H^1_D(\Omega;\mathbb{R}^d)}&\quad l(v)+\mu\left[l(v)-\int_\Omega u(E(v):\mathbf{C}:E(v))~\dd x\right]\\
\mbox{s.t.}   &\quad \int_\Omega u~\dd x\leq V,\nonumber
\end{align}
where $\mu>0$ represents a large penalty parameter. Since the problem \eqref{TOP_2} is (still) ill-posed, approximate regularized problems are solved in practice in order to compute a stable topological design. In this direction, besides material interpolation schemes with penalization (cf. \cite{TOP_Book}), filter methods and perimeter penalization are of particular importance (cf. \cite{TOP_Book,Filter_TOP}). If $K$ is a suitable nonnegative and smooth filter function approximating the Dirac-delta distribution, and $\varepsilon\in(0,1)$ is a small parameter to prevent the possible singularity of the equilibrium problem, then the regularized problem takes the form
\begin{align}\label{TOP_3}
& \min\limits_{(u,v)\in BV(\Omega;\{0,1\})\times H^1_D(\Omega;\mathbb{R}^d)}\quad \mu\left[l(v)-\int_\Omega (\varepsilon+(1-\varepsilon)(K\star u)) (E(v):\mathbf{C}:E(v))~\dd x\right]\nonumber\\&\qquad\qquad\qquad\qquad+l(v)+\beta~J(u)\\
&\mbox{s.t.}   \quad \int_\Omega u~\dd x\leq V,\nonumber
\end{align}
where  '$\star$' denotes convolution.

We are now able to define the functional $F: H^1_D(\Omega;\mathbb{R}^d)\rightarrow \mathbb{R}^+$ as
\[
 F(v):=(1+\mu) l(v)-\mu\varepsilon\int_\Omega E(v):\mathbf{C}:E(v)~\dd x,
\]
and the operator $G: H^1_D(\Omega;\mathbb{R}^d)\rightarrow L^2(\Omega)$ as
\[
 G(v)u:= -\mu(1-\varepsilon)\int_\Omega (K\star u)~(E(v):\mathbf{C}:E(v))~\dd x.
\]
It is easy to see that $F$ and $G$ satisfy the above conditions, and hence, minimizers of the minimal compliance problem can be obtained as level sets of minimizers of the relaxed problem
\begin{align}\label{TOP_4}
 &\min\limits_{(u,v)\in BV(\Omega;[0,1])\times H^1_D(\Omega;\mathbb{R}^d)}\quad F(v)+\int_\Omega G(v)u~\dd x+\beta~J(u)\\
&\mbox{s.t.}  \quad \int_\Omega u~\dd x\leq V.\nonumber
\end{align}

\section{Conclusion}

The exact relaxation framework considered in this paper allows to relax a class of minimization problems for binary valued unknowns in the space of functions of bounded variation to a convex, constrained minimization problem for a continuously valued unknown function. The solution of the original problem is almost surely gained back by computing level sets of the solution to the relaxed problem. This procedure is stable under certain data perturbations and allows to account for volume constraints, which are important, for instance, in minimal compliance problems in topology optimization. 

As it has been already observed for classical total variation regularization with $L^1$- or $L^2$-data fidelity terms (\cite{MR2559163,semismooth2,MR2219285})
inexact variants of semismooth Newton methods allowing for rather inexact solutions of the associated linear systems are efficient solvers also for problems of the underlying type. This is demonstrated in this work for examples of binary image denoising and multi-phase image segmentation.

The exact relaxation framework (in analysis and numerics) allows to treat various application classes. Besides image processing (denoising, segmentation, labelling), the design of composite membranes and the computation of topological structures with minimal compliance fall into this category.

\bibliographystyle{plain}
\bibliography{reference}

\begin{thebibliography}{10}

\bibitem{AcVo94}
R.~Acar and C.~R. Vogel.
\newblock Analysis of total variation penalty methods.
\newblock {\em Inverse Problems}, 10:1217--1229, 1994.

\bibitem{TOP_Book2}
G.~Allaire.
\newblock {\em Shape Optimization by the Homogenization Method}.
\newblock Springer, New York, 2002.

\bibitem{MR1079985}
L.~Ambrosio.
\newblock Metric space valued functions of bounded variation.
\newblock {\em Ann. Scuola Norm. Sup. Pisa Cl. Sci. (4)}, 17(3):439--478, 1990.

\bibitem{ML_IJCV}
E.~Bae, J.~Yuan, and X.-C. Tai.
\newblock Global minimization for continuous multiphase partitioning problems
  using a dual approach.
\newblock {\em International Journal of Computer Vision}, 92(1):112--129, 2011.

\bibitem{TOP_Book}
M.~P. Bends\o{e} and O.~Sigmund.
\newblock {\em Topology Optimization}.
\newblock Springer, Berlin, 2002.

\bibitem{Filter_TOP}
B.~Bourdin.
\newblock Filters in topology optimization.
\newblock {\em International Journal for Numerical Methods in Engineering},
  50:2143--2158, 2001.

\bibitem{exactpnlty1}
J.~V. Burke.
\newblock An exact penalization viewpoint of constrained optimization.
\newblock {\em SIAM Journal on Control and Optimization}, 29:968--998, 1991.

\bibitem{TV1}
A.~Chambolle and P-L. Lions.
\newblock Image recovery via total variation minimization and related problems.
\newblock {\em Numerische Mathematik}, 76:167--188, 1997.

\bibitem{analysisL1TV1}
T.~F. Chan and S.~Esedo\={g}lu.
\newblock Aspects of total variation regularized {L}$^1$ function
  approximation.
\newblock {\em SIAM Journal on Applied Mathematics}, 65:1817--1837, 2005.

\bibitem{CEM-2006}
T.~F. Chan, S.~Esedo\={g}lu, and M.~Nikolova.
\newblock Algorithms for finding global minimizers of image segmentation and
  denoising models.
\newblock {\em SIAM Journal on Applied Mathematics}, 66(5):1632--1648, 2006.

\bibitem{Chan_Vese}
T.~F. Chan and L.~Vese.
\newblock Active contours without edges.
\newblock {\em IEEE Transactions on Image Processing}, 10(2):266--277, 2001.

\bibitem{OD_1}
S.~Chanillo, D.~Grieser, M.~Imai, K.~Kurata, and I.~Onishi.
\newblock Symmetry breaking and other phenomena in the optimization of
  eigenvalues for composite membranes.
\newblock {\em Communications in Mathematical Physics}, 214:315--337, 2000.

\bibitem{OD_2}
S.~J. Cox and J.~R. McLaughlin.
\newblock Extremal eigenvalue problems for composite membranes, i and ii.
\newblock {\em Applied Mathematics and Optimization}, 1990.

\bibitem{ML_Chapter}
D.~Cremers, T.~Pock, K.~Kolev, and A.~Chambolle.
\newblock Convex relaxation techniques for segmentation, stereo and multiview
  reconstruction.
\newblock In {\em Markov Random Fields for Vision and Image Processing}. MIT
  Press, 2011.

\bibitem{DeZo01}
M.C. Delfour and J.-P. Zol{\'e}sio.
\newblock {\em Shapes and geometries. Analysis, differential calculus, and
  optimization}.
\newblock SIAM, Philadelphia, 2001.

\bibitem{MR2559163}
Y.~Dong, M.~Hinterm{\"u}ller, and M.~Neri.
\newblock An efficient primal-dual method for {${\rm L}^1{\rm TV}$} image
  restoration.
\newblock {\em SIAM J. Imaging Sci.}, 2(4):1168--1189, 2009.

\bibitem{Analysis}
I.~Ekeland and R.~T\'{e}mam.
\newblock {\em Convex Analysis and Variational Problems}.
\newblock Classics in Applied Mathematics 28, SIAM, Philadelphia, 1999.

\bibitem{coarea}
W.~H. Fleming and R.~Rishel.
\newblock An integral formula for the total gradient variation.
\newblock {\em Archiv der Mathematik}, 11(1):218--222, 1960.

\bibitem{BVBook2}
E.~Giusti.
\newblock {\em Minimal Surfaces and Functions of Bounded Variation}.
\newblock Birkh\"{a}user, Boston, 1984.

\bibitem{semismooth2}
M.~Hinterm\"{u}ller, K.~Ito, and K.~Kunisch.
\newblock The primal-dual active set strategy as a semismooth {N}ewton method.
\newblock {\em SIAM J. Optim.}, 13:865--888, 2002.

\bibitem{semismooth}
M.~Hinterm\"{u}ller and K.~Kunisch.
\newblock Total bounded variation regularization as bilaterally constrained
  optimization problem.
\newblock {\em SIAM Journal on Applied Mathematics}, 64:1311--1333, 2004.

\bibitem{MR2219285}
M.~Hinterm{\"u}ller and G.~Stadler.
\newblock An infeasible primal-dual algorithm for total bounded variation-based
  inf-convolution-type image restoration.
\newblock {\em SIAM J. Sci. Comput.}, 28(1):1--23, 2006.

\bibitem{exactpnlty2}
J.-B. Hiriart-Urruty and C.~Lemarechal.
\newblock {\em Convex Analysis and Minimization Algorithms}.
\newblock Springer, 1993.

\bibitem{ML_SIIMS11}
J.~Lellmann and C.~Schn\"{o}rr.
\newblock Continuous multiclass labeling approaches and algorithms.
\newblock {\em SIAM Journal on Imaging Sciences}, 4(4):1049--1096, 2011.

\bibitem{Mumford-Shah}
D.~Mumford and J.~Shah.
\newblock Approximation by piecewise smooth functions and associated
  variational problems.
\newblock {\em Communications on Pure and Applied Mathematics}, 42:577--685,
  1989.

\bibitem{ML_SIIMS10}
T.~Pock, D.~Cremers, H.~Bischof, and A.~Chambolle.
\newblock Global solutions of variational models with convex regularization.
\newblock {\em SIAM Journal on Imaging Sciences}, 3(4):1122--1145, 2010.

\bibitem{Potts}
R.~B. Potts.
\newblock Some generalized order-disorder transformations.
\newblock In {\em Proceeding of the Cambridge Philosophical Society},
  volume~48, pages 106--109, 1952.

\bibitem{TVmodel}
L.~I. Rudin, S.~Osher, and E.~Fatemi.
\newblock Nonlinear total variation based noise removal algorithms.
\newblock {\em Physica D}, 60:259--268, 1992.

\bibitem{funanalysis}
W.~Rudin.
\newblock {\em Functional Analysis}.
\newblock International Series in Pure and Applied Mathematics.
  Mc{G}raw-{H}ill, 2nd edition, 1991.

\bibitem{MS_Burger}
A.~Sawatzky, D.~Tenbrinck, X.~Jiang, and M.~Burger.
\newblock A variational framework for region-based segmentation incorporating
  physical noise models.
\newblock CAM Report 11-81, April 2012.

\bibitem{ML_CVPR}
E.~Strekalovskiy, A.~Chambolle, and D.~Cremers.
\newblock A convex representation for the vectorial {Mumford-Shah} functional.
\newblock In {\em IEEE Conference on Computer Vision and Pattern Recognition
  (CVPR)}, Providence, Rhode Island, June 2012.

\bibitem{ML_CV}
L.~A. Vese and T.~F. Chan.
\newblock A multiphase level set framework for image segmentation using the
  {M}umford and {S}hah model.
\newblock {\em International Journal of Computer Vision}, 50(3):271--293, 2002.

\bibitem{Zach_ML}
C.~Zach, D.~Gallup, J.-M. Frahm, and M.~Niethammer.
\newblock Fast global labeling for real-time stereo using multiple plane
  sweeps.
\newblock In {\em Vision, Modeling and Visualization Workshop (VMV)}, 2008.

\end{thebibliography}
\end{document}